\documentclass[leqno,11pt]{amsart}

\usepackage{amssymb}
\usepackage{amsmath}
\usepackage{enumerate}
\usepackage[pdftex]{graphicx}
\usepackage{graphicx, color}

\def\R{\mathbb{R}}

\DeclareMathOperator{\sech}{sech}

\newtheorem{theorem}{Theorem}[section]
\newtheorem{proposition}[theorem]{Proposition}
\newtheorem{corollary}[theorem]{Corollary}

\theoremstyle{definition}

\newtheorem{remark}[theorem]{Remark}
\newtheorem{example}[theorem]{Example}

\numberwithin{equation}{section}

\begin{document}

\title[A new approach to rotational Weingarten surfaces]{A new approach to \\ rotational Weingarten surfaces}

\author[P. Carretero]{Paula Carretero}
\address{Departamento de Matem\'aticas \\
Universidad de Ja\'{e}n \\
23071 Ja\'{e}n, Spain.}
\email{pch00005@red.ujaen.es}

\author[I. Castro]{Ildefonso Castro}
\address{Departamento de Matem\'{a}ticas \\
Universidad de Ja\'{e}n \\
23071 Ja\'{e}n, Spain and IMAG, Instituto de Matem\'aticas de la Universidad de Granada.}
\email{icastro@ujaen.es}

%\thanks{Research partially supported by 
%a MEC-FEDER grant MTM2017-89677-P. Research of the third named author partially supported by a MECD grant FPU16/03096.
%}



\date{}

\begin{abstract}
Weingarten surfaces are those whose principal curvatures satisfy a functional relation,
whose set of solutions is called the curvature diagram or the W-diagram of the surface.
Making use of the notion of \textit{geometric linear momentum} of a plane curve, we propose a new approach to the study of rotational Weingarten surfaces in Euclidean 3-space. Our contribution consists of reducing any type of Weingarten condition on a rotational surface to a first order differential equation on the momentum of the generatrix curve. In this line, we provide two new classification results involving a cubic and an hyperbola in the W-diagram of the surface characterizing, respectively, the non-degenerated quadric surfaces of revolution and the \textit{elasticoids}, defined as the rotational surfaces generated by the rotation of the Euler elastic curves around their directrix line. 

As another application of our approach, we deal with the problem of prescribing mean or Gauss curvature on rotational surfaces in terms of arbitrary continuous functions depending on distance from the surface to the axis of revolution. As a consequence, we provide simple new proofs of some classical results concerning rotational surfaces, like Euler's theorem about minimal ones, Delaunay's theorem on constant mean curvature ones, and Darboux's theorem about constant Gauss curvature ones. %We also give some uniqueness results of some featured rotational surfaces.
\end{abstract}

\maketitle

%\tableofcontents

\section{Introduction}

 Following \cite{Ch45}, the \textit{Weingarten surfaces} are those whose principal curvatures $\kappa_1$ and $\kappa_2$ satisfy a certain functional relation $\Phi(\kappa_1,\kappa_2) = 0$. 
 The set of solutions of this equation is also called the \textit{curvature diagram} or the $W$\textit{-diagram} of the surface (see \cite{H51}).
 These surfaces were introduced by Weingarten in \cite{W61} and their study plays an important role in classical Differential Geometry (see e.g.\ \cite{H51} and \cite{V59}). Applications of Weingarten surfaces on computer aided design can be found in \cite{BG96}.

In particular, the ones satisfying a linear relation $a \kappa_1 + b\kappa_2 = c$, $a^2+b^2\neq 0$, $c\in \R$, are called \textit{linear Weingarten surfaces}. In this case, the W-diagram is contained in a straight line or degenerates to one point.
This class of Weingarten surfaces include umbilical surfaces, isoparametric surfaces, constant mean curvature and minimal surfaces or those surfaces where one of the principal curvatures is constant.
On the other hand, some types of thin axial symmetric shells subjected to uniform normal pressure are modelled on rotational surfaces whose principal curvatures obey some specific quadratic relations (see e.g.\ \cite{PHM19} and \cite{PM20}). Concretely, the W-diagram is contained in a certain parabola in this case. The general case that the curvature diagram is a standard parabola was studied in \cite{KS05}.

Among the invariant surfaces, precisely the rotational ones are probably the most studied surfaces in Euclidean 3-space. Perhaps the main reason could be they are a nice family where considering firstly interesting geometric properties to attach on any surface, because its geometry can be controlled by the geometry of the generatrix curve. As an illustration, we can mention the classical theorems of Euler \cite{E44}, Delaunay \cite{D41} and Darboux \cite{D90} classifying minimal, constant mean curvature and constant Gauss curvature rotational surfaces in 1744, 1841 and 1890 respectively.
However, the complete classification of rotational linear Weingarten surfaces  has not been achieved surprisingly until 2020 in \cite{LP20a}. 
Some other interesting rotational Weingarten surfaces have been also recently studied in \cite{LP20b} and \cite{LP20c}. We refer to \cite{KS05} (and references therein) for the study of closed rotational Weingarten surfaces satisfying $\kappa_1= c \, \kappa_2^{2n+1}$, generalizing results of Hopf \cite{H51} when $n=0$ and the case of ellipsoids of revolution when $n=1$.
These may be good reasons why rotational surfaces are probably one of the main classes of Weingarten surfaces and continue to deserve attention. We propose in this paper a new approach for their study,
inspired mainly by \cite{CCI16}.

In order to explain our focus, we start recalling that plane curves are uniquely determined, up to rigid motions, by its intrinsic equation giving its curvature $\kappa $ as a function of its arc-length. It is well known that one needs three quadratures in the integration process. In \cite{S99}, David A.\ Singer considered a different sort of problem, proposing
to determine a plane curve if its curvature is given in terms of its position.
Probably, the most interesting solved case in this setting corresponds to the Euler elastic curves (see \cite{S08}), 
whose curvature is proportional to one of the coordinate functions,
e.g.\ $ \kappa(x, z) = c\, x $ for curves in the $xz$-plane.
Motivated by the above question and by the classical elasticae, the authors studied in \cite{CCI16}
the plane curves whose curvature depends on the distance to a line (say the $z$-axis and so $\kappa=\kappa(x)$) and in \cite{CCIs17}
the plane curves whose curvature depends on the distance from a point (say the origin, and so $\kappa=\kappa(r)$, $r =\sqrt{x^2+z^2}$) requiring in both cases the computation of three quadratures too. They also considered the analogous problems in Lorentz-Minkowski plane in \cite{CCIs18} and \cite{CCIs20a}.

The key tool for the study of plane curves whose curvature depends on distance to a line is the notion of \textit{geometric linear momentum} of a plane curve (see \cite{CCIs20b} or Section \ref{Sect2}). It is a smooth function associated to any plane curve, which completely determines it (up to a family of distinguished translations) in relation with its relative position with respect to a fixed line (see Corollary \ref{cor;key}). Geometrically, the geometric linear momentum controls the angle of the Frenet frame of the curve with this fixed line and receives that name because, in physical terms, it can be described as the linear momentum (with respect to the fixed line) of a particle of unit mass with unit speed and trajectory the path of the plane curve. Moreover, it can be interpreted as an anti-derivative of the curvature of the curve when this is expressed as a function of the distance to the fixed line.
 
Therefore it seems to be natural that when one deals with rotational surfaces, generated by the rotation of a plane curve (the generatrix curve) around a coplanar fixed line (the axis of revolution) the geometric linear momentum of the generatrix with respect to the axis of revolution plays a predominant role to control the geometry of the rotational surface. We show it in Corollary \ref{cor:deter}, proving that any rotational surface is uniquely determined, up to translations along the axis of revolution, by the geometric linear momentum of the generatrix curve. This main result is confirmed  when we study the geometry of a rotational surface since both its first and second fundamental forms can be expressed only in terms of the geometric linear momentum and, of course, the (non-constant) distance from the surface to the axis of revolution.
This allows our main contribution in the paper which consists of reducing any type of Weingarten condition on a rotational surface to a first order differential equation on the momentum of the generatrix curve (see Section \ref{Sect3}). We illustrate this procedure analysing under our optics the two types of linear Weingarten surfaces one can find in literature (see Subsections \ref{Sect linear} and \ref{Sect special W-linear}) and we emphasize two new classification results involving a cubic and an hyperbola in the W-diagram of the rotational Weingarten surface (see Theorems \ref{th:W-cubic} and \ref{th:W-hyperbola} respectively) characterizing, respectively, the non-degenerated quadric surfaces of revolution and the \textit{elasticoids} (defined as the rotational surfaces generated by the rotation of the Euler elastic curves around their directrix line). 

As another application, we deal in Section \ref{Sect4} with the problem of prescribing mean or Gauss curvature on rotational surfaces in terms of arbitrary continuous functions depending on distance from the surface to the axis of revolution, providing in Theorem \ref{th:Prescribe H KGauss} one-parameter families of rotational surfaces with prescribed mean or Gaussian curvature. We point out that Kenmotsu \cite{K80} 
constructed a 3-parameter family of surfaces of revolution admitting $H=H(s)$ as the mean curvature for a
given continuous function $H(s)$, $s$ being the arc parameter of the generatrix curve.

As a consequence of Theorem \ref{th:Prescribe H KGauss}, we provide simple new proofs of the above mentioned classical results concerning rotational surfaces like Euler's theorem about minimal ones (see Corollary \ref{cor:Euler Th}), Delaunay's theorem on constant mean curvature ones (see Corollary \ref{cor:Delaunay Th}), and Darboux's theorem about constant Gauss curvature ones (see Corollary \ref{cor:Darboux Th}).

%Finally, in Section \ref{Sect5}, we also give some uniqueness results of some featured rotational surfaces.

\section{The geometric linear momentum of a plane curve}\label{Sect2}
	
We introduce a smooth function associated to any plane curve, which completely determines it (up to a family of distinguished isometries) in relation with its relative position with respect to a fixed line.

Let $\alpha  : I\subseteq \mathbb R \rightarrow \mathbb R^2 $ be a regular plane curve parametrized by arc-length; that is, $|\dot \alpha(s)|=1, \, \forall s\in I $, where $I$ is some interval in $\mathbb R$. Here $\, \dot{}\, $ means derivation with respect to $s$. 
We denote by $\langle \cdot , \cdot \rangle$ the usual inner product in $\mathbb R^2$. Let $T:=\dot{\alpha}$
be the unit tangent vector to the curve $\alpha$ and let $N:=i\dot{\alpha}$ be the vector orthogonal to $T$ such that
the frame $(T,N)$ is positively oriented. The Frenet equations of $\alpha $ are given by
\begin{equation}\label{eq:FrenetEqs}
\dot T(s)= \kappa (s)  N(s), \quad \dot N(s)= -\kappa (s) T(s),
\end{equation}
where $\kappa=\kappa(s)=\det (\dot \alpha (s), \ddot \alpha (s))$ is the (signed) curvature of $\alpha$. Up to sign, $\kappa=\pm \dot \theta$, where $\theta $ can be chosen as the angle $\angle (T,(1,0))$.

We are interested in the geometric condition that the curvature of $\alpha$ depends on the distance to a fixed line of $\R^2$. 
We choose Cartesian coordinates in $\R^2 $ and we write $\alpha =(x,z)$.
Thus it is enough to study the condition $\kappa=\kappa(x)$, since $x$ represents the signed distance to the $z$-axis. 

A key role is played then by the \textit{geometric linear momentum} of the plane curve $\alpha $ (with respect to the $z$-axis).
At a given point $\alpha (s)$ on the curve, it is defined by 
$$\mathcal K (s) := \langle T(s),(0,1) \rangle  = \sin \theta (s). $$
Geometrically, $\mathcal K$ controls the angle of the Frenet frame of the curve with the coordinate axes. 

We can also observe a physical interpretation if we write $\mathcal K $ in Cartesian coordinates. 
Since $\alpha =(x,z)$, then $T =(\dot x, \dot z)$ and so the linear momentum $\mathcal K$ at $\alpha (s)$ is given by
\begin{equation} \label{momentum K}
\mathcal K(s)= \dot z (s).
\end{equation}
Hence, in physical terms, $\mathcal K= \mathcal K (s)$ may be described as the linear momentum (with respect to the $z$-axis) of a particle of unit mass with unit speed and trajectory $\alpha (s)$.
We point out that $\mathcal K$ assumes values in $[-1,1]$ and it is well defined, up to the sign, depending on the orientation of $\alpha$.

\begin{remark}\label{no ppa}
If the plane curve $\alpha=(x,z)$ is not necessarily parametrized by arc length, i.e. $\alpha = \alpha (t)$, $t$ being any parameter, one can computes the geometric linear momentum $\mathcal K = \mathcal K (t)$ by means of
$$
\mathcal K (t)= \frac{z'(t)}{|\alpha ' (t)|},
$$
where $'$ denotes derivation respect to $t$.
\end{remark}

It is an easy exercise the computation of the geometric linear momenta of some distinguish plane curves for our purposes, using Remark \ref{no ppa} if necessary, some of them located in a determinate position with respect to $z$-axis, that we collect in the following list (see Figure \ref{fig:3curves}):

%\newpage

\begin{example}\label{ex:curves}
\noindent
\begin{enumerate}
\item Vertical lines: $\mathcal K \equiv \pm 1$.
\item Line with slope $\tan \theta_0$, $\theta_0 \in (-\pi/2, \pi/2)$, given by   $z=\tan \theta_0 \, x$: $\mathcal K \equiv \sin \theta_0 \in (-1,1)$.
\item Circle centred at $(a,0)$, $a\in \R$, and radius $R>0$: \\
$\mathcal K (s)=\cos (s/R)=  (x-a)/R=\mathcal K (x)$.
\item Catenary $x=a \cosh (z/a)$, $a>0$: $\mathcal K (x)=a/x$.
\item Cycloid of radius $R>0$ given by $x=R(1-\cos t)$, $z=R(t-\sin t)$, $t\in \R$: $\mathcal K (t)= \sin (t/2)=\sqrt x / \sqrt{2R}= \mathcal K (x)$.
\item Tractrix of height $a>0$ given by $x=a \sech t$, $z=a(t-\tanh t)$, $t\in \R$: $\mathcal K (t)=\tanh t=\sqrt{1-x^2/a^2}=\mathcal K
(x)$.
\end{enumerate}
\end{example} 

\begin{figure}[h!]
\begin{center}
\includegraphics[width=6cm]{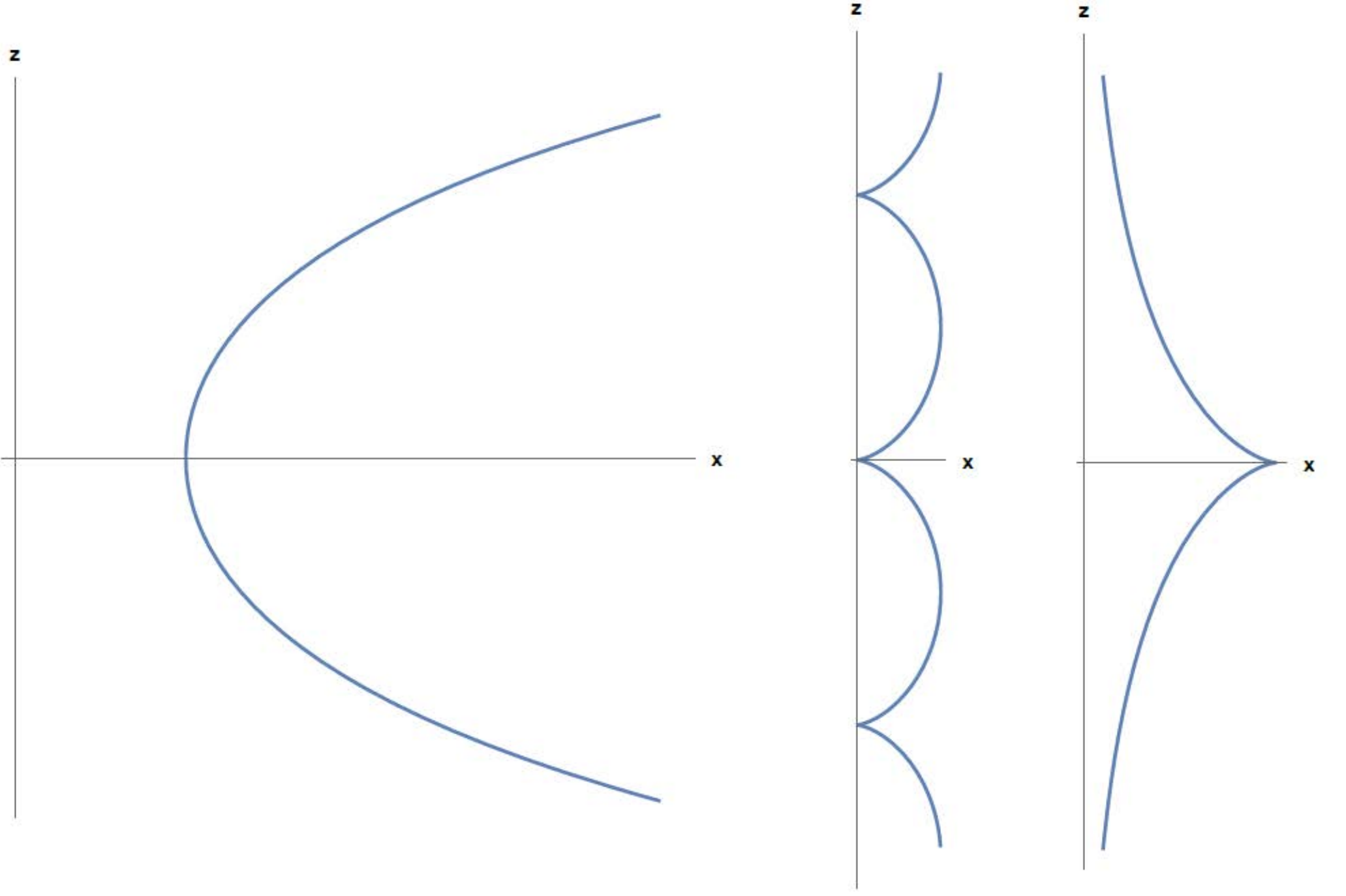}
\caption{From left to right: catenary, cycloid, tractrix.}
\label{fig:3curves}
\end{center}
\end{figure}

On the other hand, when $\alpha$ is unit-speed we have that 
\begin{equation}\label{ppa}
\dot x ^2 + \dot z ^2 =1.
\end{equation}
Using \eqref{momentum K} and \eqref{ppa}, assuming that $x$ is non constant, we get:
\begin{equation}\label{differentials}
ds = \frac{dx}{\sqrt{1-\dot z^2}}=\frac{dx}{\sqrt{1-\mathcal K^2}}, \quad dz = \mathcal K \, ds.
\end{equation}
Thus, given $\mathcal K=\mathcal K (x)$ as an explicit function, looking at \eqref{differentials} one may attempt to compute $x(s)$ and $z (s) $ in three steps: integrate to get $s=s(x)$, invert to get $x=x(s)$ and integrate to get $z =z (s)$. 

In addition, \eqref{eq:FrenetEqs} implies that $\ddot z = \kappa \dot x$. In this way, assuming that $\kappa=\kappa(x)$ and using \eqref{momentum K}, we obtain:
\begin{equation}\label{anti K}
d \mathcal K = \dot{\mathcal K} (s) ds =  \ddot z (s) ds = \kappa (x) \dot x (s)ds = \kappa (x) dx.
\end{equation}
Therefore we deduce that $\mathcal K $ can be interpreted as an anti-derivative of $\kappa= \kappa (x)$.
As a summary, we can determine by quadratures in a constructive explicit way the plane curves such that $\kappa=\kappa(x)$, in the spirit of  \cite[Theorem 3.1]{S99}.
\begin{theorem}\label{Th:k=k(x)}
Let $\kappa=\kappa(x)$ be a continuous function.
Then the problem of determining a curve $(x(s),z(s)) $, with $s$ the arc length parameter, whose curvature is $\kappa(x)$  with geometric linear momentum $\mathcal K (x)$ satisfying \eqref{anti K} ---$x$ representing the (non constant) signed distance to the $z$-axis---
is solvable by quadratures where $x(s)$ and $z(s)$ are obtained through \eqref{differentials} after inverting $s=s(x)$. 

Moreover, such a curve is uniquely determined, up to translations in the $z$-direction (and a translation of the arc parameter $s$), by the geometric linear momentum $\mathcal K(x)$.
\end{theorem}

If we focus on the determining role of the geometric linear momentum, we can rephrase the above result simply as follows: 

\begin{corollary}\label{cor;key}
Any plane curve $\alpha =(x,z)$, with $x$ non-constant, is uniquely determined by its geometric linear momentum $\mathcal K$ as a function of its distance to $z$-axis, that is, by $\mathcal K= \mathcal K(x)$. 
The uniqueness is modulo translations in the $z$-direction.
Moreover, the curvature of $\alpha $ is given by $\kappa (x)=\mathcal K' (x)$.
\end{corollary}

\begin{remark}\label{c}
If we prescribe a continuous function $\kappa=\kappa(x)$ as curvature of a plane curve, the proof of Theorem~\ref{Th:k=k(x)} offers an algorithm to recover the curve through the computation of three quadratures, following the sequence:
\begin{enumerate}[\rm (i)]
\item[\rm (i)] A one-parameter family of anti-derivatives of $\kappa (x)$:
$$
\int \! \kappa (x) dz = \mathcal K(x).
$$
\item[\rm (ii)] Arc-length parameter $s$ of $\alpha=(x,z) $ in terms of $x$, defined ---up to translations of the parameter--- by the integral:
$$
s=s(x)=\int\!\frac{dx}{\sqrt{1-\mathcal K(x)^2}},
$$
where $-1<\mathcal K(x)<1$, and
inverting $s=s(x)$ to get $x=x(s)$. 
\item[\rm (iii)]$z$-coordinate of the curve ---up to translations along $z$-axis--- by the integral:
$$
z (s)=\int \! \mathcal K(x(s))\, ds .
$$
\end{enumerate}
We note that we get a one-parameter family of plane curves satisfying $\kappa=\kappa(x)$ according to the linear momentum $\mathcal K (x)$ chosen in {\rm (i)} and verifying $\mathcal K(x)^2 < 1 $. It will distinguish geometrically the curves inside a same family by their relative position with respect to $z$-axis. We recall that we can easily recover $\kappa$ from $\mathcal K$ simply by means of $\kappa =  d\mathcal K/dx$. In some sense, $\mathcal K= \mathcal K(x)$ can be interpreted as the \textit{extrinsic equation} of $\alpha$ in this setting, in contrast to the classical intrinsic equation $\kappa=\kappa(s)$. 
\end{remark}

We now show two illustrative examples applying steps {\rm (i)-(iii)} of the algorithm described in Remark~\ref{c}:

\begin{example}[$\kappa\!\equiv\! 0$]
Then $\mathcal K \equiv c\in (-1,1)$, and easily $x(s)=\sqrt{1-c^2}\,s$, $z(s)=c\,s$. Writing $c=\sin \theta_0$, $\theta_0 \in (-\pi/2, \pi/2)$, we arrive at Example \ref{ex:curves}(2). Up to $z$-translations, we get all the (non vertical) lines in the plane.
\end{example}

\begin{example}[$\kappa \equiv k_0 >0 $]
Then $\mathcal K (x)= k_0 x + c$, and it is not difficult to obtain that $x(s)=\frac{1}{k_0}\left(\sin (k_0 s)-c\right)$, $z(s)=-\frac{1}{k_0}\cos (k_0 s)$. Writing $k_0=1/R$ we arrive at Example \ref{ex:curves}(3). Up to $z$-translations, we get all the circles with radius $R=1/k_0$ in the plane.
\end{example}

\begin{remark}\label{difficulties}
{\rm The main difficulties one can find carrying on the strategy described in Remark~\ref{c}  (or in Theorem~\ref{Th:k=k(x)}) to determine a plane curve whose curvature is $\kappa=\kappa(x)$ are the following:
\begin{enumerate}[\rm (1)]
\item The integration of $s=s(x)$:  Even in the case that $\mathcal K (x)$ were polynomial,
the corresponding integral is not necessarily elementary. For example, when $\mathcal K (x)$ is a quadratic
polynomial, it can be solved using Jacobian elliptic functions (see e.g.\ \cite{BF71}). This last case is
equivalent to $\kappa (x)$ be linear, i.e.\ $\kappa (x)= 2a x+ b$, $a\neq 0$, $b \in \R$. These are precisely the Euler elastic curves (see \cite[Section 3]{CCI16}).
\item The previous integration gives us $s=s(x)$; it is not always possible to obtain explicitly $x=x(s)$, what is necessary to determine the curve.
\item Even knowing explicitly $x=x(s)$, the integration to get $z(s)$ may be impossible to perform using elementary or known functions.
\end{enumerate}
}
\end{remark}

We finish this section with an example illustrating the difficulty (2) in Remark \ref{difficulties}.

\begin{example}\label{ex:parabolas}
Consider $\mathcal K(x)=c/\sqrt x$, $x>c^2>0$. Then $s=s(x)=\sqrt x \sqrt{x-c^2}+c^2 \ln (\sqrt x \sqrt{x-c^2})$ and it is not possible to invert to get $x=x(s)$.
However, eliminating $ds$ in Remark \ref{c}(ii) and (iii), we obtain:
\begin{equation}\label{eq:zgraph}
z=z(x)=\pm \int \frac{\mathcal K(x)dx}{\sqrt{1-\mathcal K(x)^2}}.
\end{equation}
Applying the above formula in this example, we arrive at $z^2=4c^2(x-c^2)$ that corresponds with the parabola with vertex $(c^2,0)$ and focus $(2c^2,0)$ (see Figure \ref{fig:parabola}).
\end{example}

\begin{figure}[h!]
\begin{center}
\includegraphics[width=3cm]{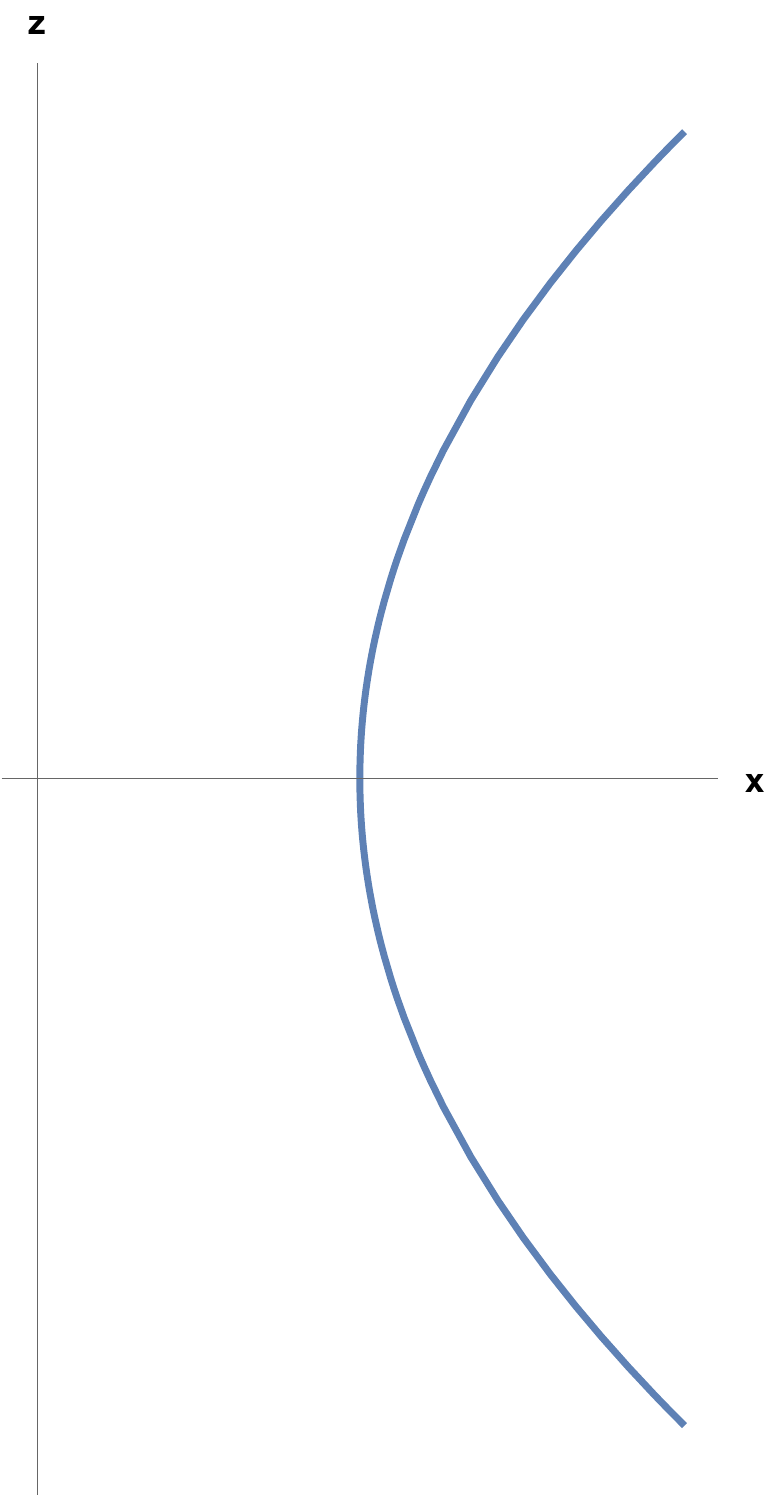}
\caption{Parabola of Example \ref{ex:parabolas}.}
\label{fig:parabola}
\end{center}
\end{figure}

\section{Rotational Weingarten  surfaces}\label{Sect3}

We start this section considering {\em rotational surfaces}, also called
{\em surfaces of revolution}. They are surfaces globally invariant under the action of any rotation around a fixed line called {\em axis of revolution}.
The rotation of a curve (called {\em generatrix}) around a fixed line generates a surface of revolution.
The sections of a surface of revolution by half-planes delimited by the axis of revolution, called {\em meridians}, are special generatrices.
The sections by planes perpendicular to the axis are circles called {\em parallels} of the surface.

We denote $S_\alpha$ the rotational surface in $\R^3$ generated by the rotation around $z$-axis of a plane curve $\alpha$ in the $xz$-plane (see Figure \ref{fig:model}). That is, $\alpha $ is the generatrix curve that we consider parametrized by arc-length, with parametric equations given by $x=x(s)>0$, $y=0$, $z= z(s)$, $s\in I \subseteq \R$. Then $S_\alpha $ is parametrized by
$$ S_\alpha \equiv X(s,\theta)=\left(x(s)  \cos \theta, x(s)\sin \theta, z(s)\right), \
(s,\theta)\in I \times (-\pi,\pi).$$
It is obvious that if we translate the generatrix curve $\alpha$ of a rotational surface $S_\alpha$  along $z$-axis, we obtain a congruent surface to $S_\alpha$. Then, as an immediate consequence of Corollary \ref{cor;key}, we deduce the following key result: 

\begin{figure}[h!]
\begin{center}
\includegraphics[height=4.5cm]{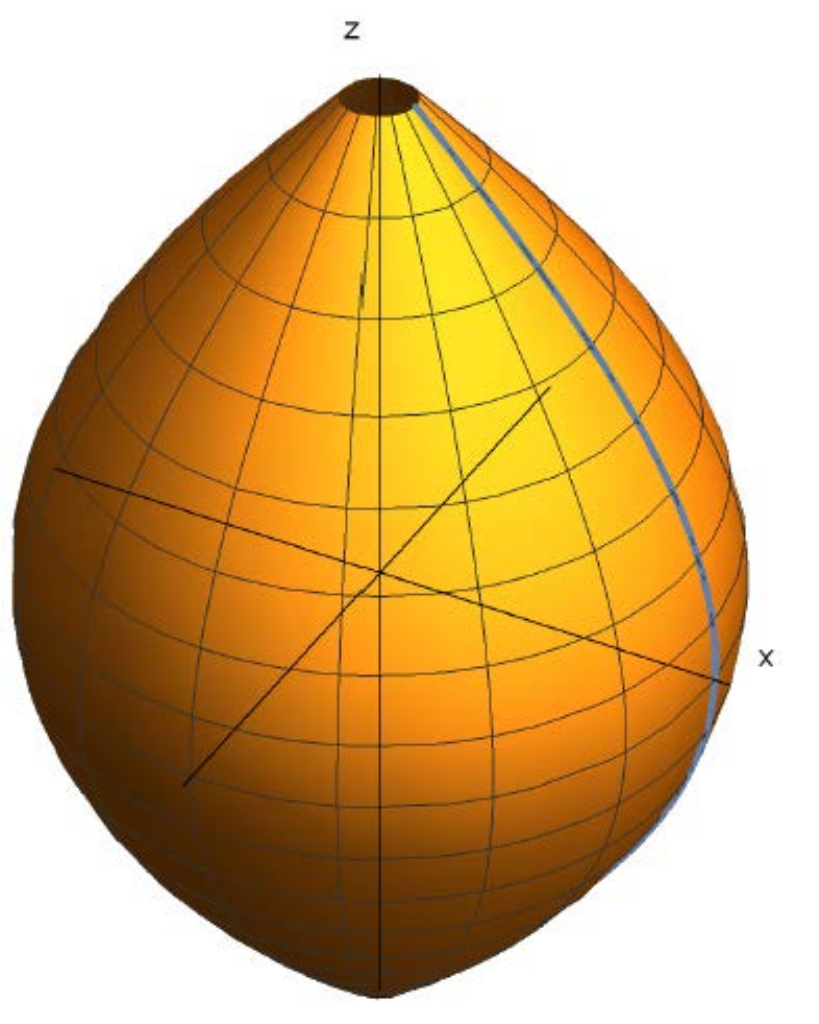}
\caption{$S_\alpha$: rotational surface generated by $\alpha=(x,z)$.}
\label{fig:model}
\end{center}
\end{figure}

\begin{corollary}\label{cor:deter}
Any rotational surface $S_\alpha$, with generatrix curve $\alpha = (x,z)$, is uniquely determined, up to $z$-translations, by the geometric linear momentum $\mathcal K=\mathcal K (x)$ of its generatrix curve, being $x$ non-constant. 
\end{corollary}

\begin{remark}\label{cylinder}
The only rotational surface excluded in Corollary \ref{cor:deter} is the {\em right circular cylinder}, corresponding to $x$ being constant. We recall that it is a flat rotational surface and its principal curvatures are $0$ (along generatrices) and $1/a$ (along parallel circles), $a>0$ being the radius of the cylinder.
\end{remark}

Looking at Examples \ref{ex:curves} and \ref{ex:parabolas} and making use of Corollary \ref{cor:deter}, we can list the following classical and conventional rotational surfaces joint with its determining geometric linear momentum. See Figure \ref{fig:ExamplesRotSurf}.

\begin{proposition}\label{ex:rot surfaces}
Up to translations in $z$-direction:
\begin{enumerate}
\item Any horizontal plane is uniquely determined by the geometric linear momentum $\mathcal K \equiv 0$.
\item The circular cone with opening $\theta_0 \in (-\pi/2, \pi/2)$, given by $x^2+y^2=\cot^2\theta_0\,z^2$, is uniquely determined by the geometric linear momentum $\mathcal K \equiv \sin \theta_0$.
\item The sphere of radius $R>0$, given by $x^2+y^2+z^2=R^2$, is uniquely determined by the geometric linear momentum $\mathcal K (x)=x/R$. 
\item The torus of revolution with major radius $|a|\neq 0$ and minor radius $R>0$, given by $(\sqrt{x^2+y^2}-a)^2+z^2=R^2$, is uniquely determined by the geometric linear momentum  $\mathcal K (x)=(x-a)/R$.
\item The catenoid of chord $a>0$, given by $x^2+y^2=a^2 \cosh^2 (z/a)$, is uniquely determined by the geometric linear momentum $\mathcal K (x)=a/x$.
\item The onducycloid of radius $R>0$, defined as the surface generated by rotation of the cycloid in Example \ref{ex:curves}(5) around its base,
is uniquely determined by the geometric linear momentum  $\mathcal K (x)= \sqrt x/\sqrt{2R}$.
\item The pseudosphere of pseudoradius $a>0$ (cf.\ \cite{F93}), defined as the surface generated by the rotation of the tractrix in Example \ref{ex:curves}(6) around its asymptote, is uniquely determined by the geometric linear momentum  $\mathcal K (x)= \sqrt{1-x^2/a^2}$.
\item The antiparaboloid of radius $c^2>0$, defined as the surface generated by the rotation of the parabola with vertex $(c^2,0)$ and focus $(2c^2,0)$, $c>0$, around its directrix line,
given by $x^2+y^2=(c^2+z^2/4c^2)^2$, is uniquely determined by the geometric linear momentum  $\mathcal K (x)= c/\sqrt{x}$.
\end{enumerate}
\end{proposition} 

\begin{figure}[h!]
\begin{center}
\includegraphics[height=4.5cm]{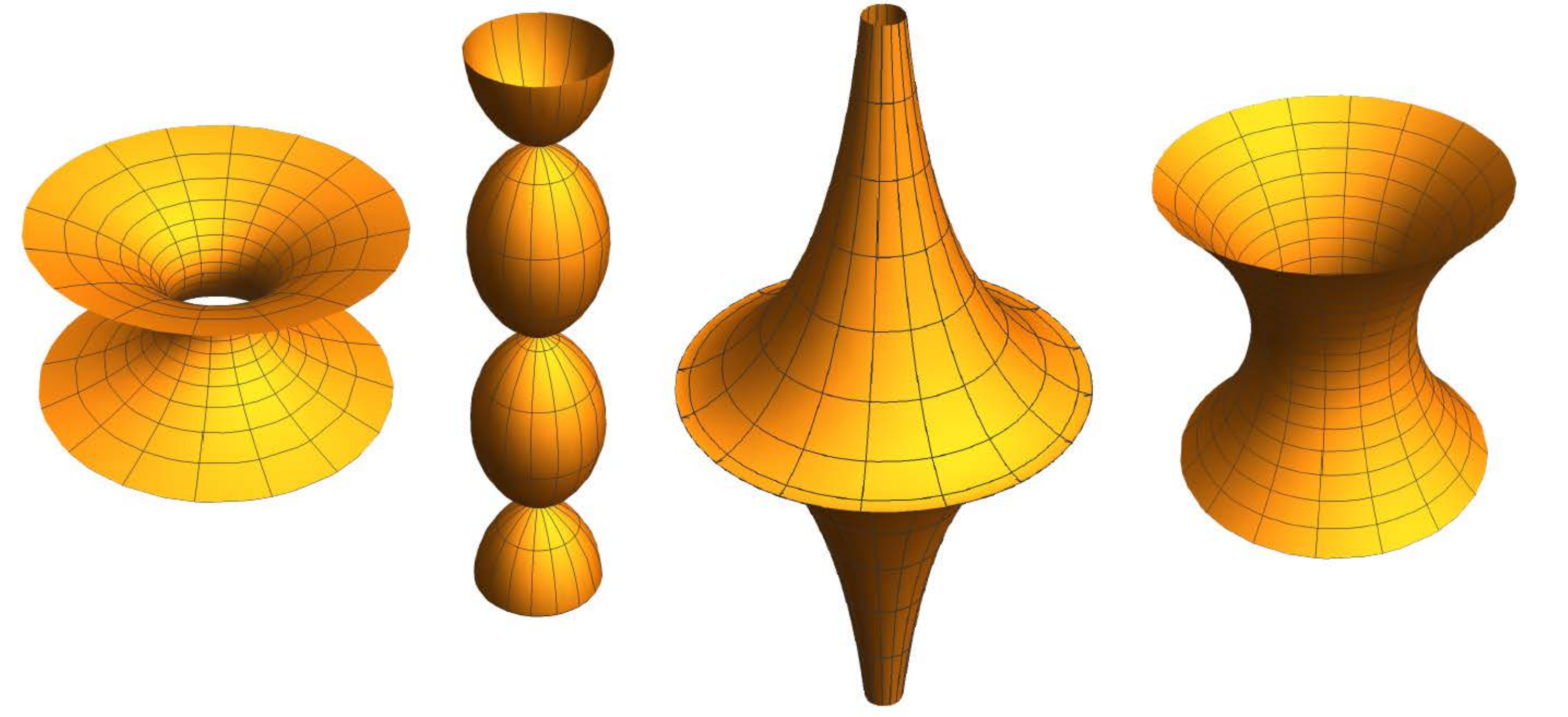}
\caption{Rotational surfaces. From left to right: catenoid, onducycloid, pseudosphere, antiparaboloid.}
\label{fig:ExamplesRotSurf}
\end{center}
\end{figure}

We can confirm the result established in Corollary \ref{cor:deter} when we study the geometry of $S_\alpha$ through its first and second fundamental forms, $I$ and $II$, since a direct computation, using that $\kappa (x)=\mathcal K ' (x) $, shows that both can be expressed only in terms of the geometric linear momentum $\mathcal K$ and, of course, the non constant distance $x$ from the surface to the axis of revolution:
 $$ I\equiv ds^2 + x^2 d\theta ^2, \quad II\equiv \mathcal K ' (x) ds^2 + x  \mathcal K (x) d\theta^2. $$

Therefore we get the following expressions for the principal curvatures $\kappa_1$ and $\kappa_2$, whose curvature lines are the meridians and parallels respectively of the rotational surface $S_\alpha$:
\begin{equation}\label{eq:prin curv}
\kappa_1\equiv k_{\text m}= \mathcal K ' (x), \quad \kappa_2\equiv k_{\text p}= \frac{\mathcal K (x)}{x}.
\end{equation}
Thus, the mean curvature $H$ of $S_\alpha $ is given by
\begin{equation}\label{eq:H}
2H=\mathcal K ' (x)+\frac{\mathcal K (x)}{x},
\end{equation}
and the Gauss curvature $K_{\text G}$ of $S_\alpha $ is given by
\begin{equation}\label{eq:KGauss}
K_{\text G}=\frac{\mathcal K (x) \mathcal K ' (x) }{x}.
\end{equation}

Now we can pay attention to {\em rotational Weingarten surfaces}. Weingarten surfaces must satisfy a functional relation between their principal curvatures.
In the case of rotational surfaces, the principal curvatures are reached along meridians and parallel of $S_\alpha$ and, from \eqref{eq:prin curv}, it is clear that rotational surfaces constitute a distinguish class of Weingarten surfaces. 
For example, if $k_{\text m}= \mathcal K ' (x)=\kappa (x)$ is invertible, from \eqref{eq:prin curv} we arrive at $k_{\text p}=\mathcal K (\kappa^{-1}(k_{\text m}))/\kappa^{-1}(k_{\text m})$. In general,
we just simply write $\Phi(k_{\text m}, k_{\text p})=0$. But, taking into account \eqref{eq:prin curv}, we easily deduce that the above functional relation translates into a first order differential equation for the geometric linear momentum $\mathcal K=\mathcal K (x)$ determining $S_\alpha$ according Corollary \ref{cor:deter}: 
$$ \hat \Phi (x,\mathcal K  (x),\mathcal K ' (x))=0. $$
If the above equation can be rewritten as $\mathcal K ' (x)=F(x,\mathcal K  (x))$ or, equivalently $k_{\text m}=\hat F (k_{\text p}) $,
it follows that locally any equation of this type with an arbitrary continuous function $F$ (or $\hat F$) admits a solution.

We can illustrate the above simple reasoning analysing some interesting Weingarten-type conditions.

\subsection{Rotational surfaces  with some constant principal curvature}\label{Sect cte}
We distinguish two cases. 
First, if some principal curvature of $S_\alpha$ is null, we have that $k_{\text m} k_{\text p}=0$. If $k_{\text m}=0$, using \eqref{eq:prin curv}, we get that $\mathcal K\equiv c$, $c\in (-1,1)$. Then, using Corollary \ref{cor:deter}, Remark \ref{cylinder} and Proposition \ref{ex:rot surfaces}, we conclude that $S_\alpha $ is a \textit{circular cylinder}, a \textit{plane} (when $c=0$) or a \textit{circular cone} (when $c\neq 0$). And if $k_{\text p}=0$, using \eqref{eq:prin curv}, we get that $\mathcal K\equiv 0$ and $S_\alpha $ is a plane. 

We consider now the case that some principal curvature is a nonzero constant. If $k_{\text m}=a\neq 0$, using \eqref{eq:prin curv}, we get that $\mathcal K(x)=ax+c$, $c\in \R$. Then, using Proposition \ref{ex:rot surfaces}, we conclude that $S_\alpha $ is a \textit{sphere} when $c=0$ or a \textit{torus of revolution} when $c\neq 0$. And if $k_{\text p}=a\neq 0$, using \eqref{eq:prin curv}, we get that $\mathcal K(x)=ax$ and $S_\alpha $ is a sphere or a \textit{circular cylinder} (see Remark \ref{cylinder}).

\subsection{Linear rotational Weingarten surfaces}\label{Sect linear}
They are defined by the linear relation		
$a \, k_{\text m}+b\, k_{\text p}=c, \ a,b,c\in\R, \ a^2+b^2\neq 0. $
We can assume $a\neq 0$ (see Subsection \ref{Sect cte} if $a=0$) and we just write 
\begin{equation}
\label{eq:linW}
k_{\text m}=p \, k_{\text p}+\, q, \, p,q\in\R, \, p\neq 0.
\end{equation}
These surfaces have been recently classified in \cite{LP20a} by means of a qualitative study, providing closed (embedded and not embedded) surfaces
and periodic (embedded and not embedded) surfaces with a geometric behaviour similar to Delaunay surfaces (\cite{D41}). 
In \cite{LP20a} there is a necessary distinction of cases according to $p=1$ or $p\neq 1$.
In fact, when $p\neq 1$, the generatrix curves are $\frac{p}{p-1}$- elastic curves (see \cite{LP21} for a description of them) generalizing classical elastic curves corresponding to $p=2$ (see \cite{LS84} and \cite{MO03}).
Under our optics, \eqref{eq:linW} translates into the linear o.d.e.\ $\mathcal K ' (x)= p \, \mathcal K (x) / x + q$, that it is easy to solve. Using Corollary \ref{cor:deter}, both above mentioned families are uniquely determined, up to $z$-translations, by the following geometric linear momenta:
$$p\neq 1:\quad\mathcal{K}(x)=\frac{q \,x}{1-p}+c \,x^{p}, \, c\in\R,$$
and		
$$p=1:\quad\mathcal{K}(x)=q \,x\ln x+c \,x, \, c\in\R.$$
This can be a reasonable explanation of the commented distinction of cases in \cite{LP20a}.

The case $p=1$, $q=0$ in \eqref{eq:linW}, i.e.\ $k_{\text m}=k_{\text p}$, leads obviously to the only umbilical surfaces (planes and spheres) since then $\mathcal K(x)=c\,x$, $c\in \R$ (see Proposition \ref{ex:rot surfaces}(1),(3)).

We pay now attention to the case $q=0$ in \eqref{eq:linW}, i.e.\ $k_{\text m}=p\, k_{\text p}$, which leads to $\mathcal{K}(x)=c\, x^{p}$. 
Under our optics, we can localize distinguished rotational surfaces satisfying $k_{\text m}=p\, k_{\text p}$ in this subclass of linear Weingarten surfaces. 
Using Proposition \ref{ex:rot surfaces}, we get \textit{catenoids} if $p=-1$, \textit{onducycloids} if $p=1/2$ and \textit{antiparaboloids} if $p=-1/2$. 

\subsection{``Linear'' rotational Weingarten surfaces}\label{Sect special W-linear}
Some authors define linear Weingarten surfaces as those ones such that a linear combination
of its mean curvature $H$ and Gauss curvature $K_{\text G}$ is constant (see e.g.\ \cite{GMM03}):
\begin{equation}\label{eq:W-quasilinear}
a \, K_{\text G} +2b \, H + c=0, \ a^2+b^2 \neq 0, c\in \R.
\end{equation}
They are also referred as special Weingarten surfaces in \cite{L08}.
They include CMC (constant mean curvature) surfaces when $a=0$ and CGC (constant Gauss curvature) surfaces when $b=0$.
The closed ones were studied in \cite{Ch45} and \cite{H51} among others. 
In \cite{RS94}, the authors studied properly embedded surfaces in $\R^3$ satisfying $a \, K_{\text G} +2b \, H=1$, where $a$ and $b$ are positive.
For $a\neq 0$, the curvature diagram corresponding to \eqref{eq:W-quasilinear} is given by the rectangular hyperbola
\begin{equation}\label{eq:rect hyperbola}
\left( k_{\text m} + \frac{b}{a}  \right) \left( k_{\text p} + \frac{b}{a}  \right) = \frac{b^2-ac}{a^2}.
\end{equation}
Using \eqref{eq:prin curv}, \eqref{eq:rect hyperbola} translates into the o.d.e.
$$
(c\,x+b\,\mathcal{K}) dx+(b\,x+a\,\mathcal K) d\mathcal K=0,
$$
that it is exact. Its implicit solution is given by
$$
a \, \mathcal K (x)^2 +2\, b \, x\, \mathcal K (x) + c \, x^2 =d \in \R,
$$
which leads to 
$$
\mathcal K (x)=\frac{1}{a}\left( -b\, x \pm \sqrt{(b^2-ac)x^2+ad} \right)
$$
when $a\neq 0$. If $c=d=0$, we recover the trivial cases corresponding to the plane (when $b=0$) and the sphere of radius $\left| \frac{-a}{2b}\right|$ (when $b\neq 0$).

%\subsection{Quadratic  rotational Weingarten surfaces}\label{Sect quadratic}
%They can be defined by some of the two possible quadratic relations	
%$k_{\text m}^2=a\, k_{\text p}$ or $k_{\text p}^2=b\, k_{\text m}$, $a,b>0$. 
%The first one leads to $\mathcal K ' (x)^2=a \mathcal K(x)/x$, whose solution is given by
%$$\mathcal{K}(x)=(\sqrt{a\,x}+c)^2, \, c\in\R.$$
%The second one translates into $(\mathcal K(x)/x)^2= b \, \mathcal K ' (x)^2$, whose solution is
%$$\mathcal{K}(x)=\frac{b\,x}{1+b\,c\,x}, \, c\in\R.$$
%In both cases, when $c=0$ we arrive at the sphere of radius $R=1/a$ or $R=1/b$.

\subsection{Cubic rotational Weingarten surfaces}\label{Sect cubic}

It is known (see \cite{BG94}) that the ellipsoid of revolution 
\begin{equation}\label{eq:ellipsoid}
\frac{x^2+y^2}{a^2}+\frac{z^2}{b^2}=1
\end{equation}
satisfies the relation
\begin{equation}\label{eq:W cubic +1}
k_{\text m}=\frac{a^4}{b^2}\, k_{\text p}^3.
\end{equation}
In \cite{KS05}, it is also proved that any \textit{closed} surface of revolution satisfying $k_{\text m}=c \, k_{\text p}^3$, for any \textit{positive} constant $c>0$, is congruent to some ellipsoid of revolution. The aim of this section is to generalize this result using our local approach to the study of rotational Weingarten surfaces.

For our purposes, recalling Corollary \ref{cor:deter}, we need to compute the geometric linear momentum of the ellipsoid \eqref{eq:ellipsoid}.
We parametrize the generatrix semiellipse by $x=a \cos t $, $z= b \sin t$, $t \in [-\pi/2,\pi/2]$, and using Remark \ref{no ppa}, it is not difficult to get that
\begin{equation}\label{eq:K ellipsoid}
\mathcal K (x) = \frac{bx}{\sqrt{a^4-(a^2-b^2)x^2}}.
\end{equation} 
Then, using \eqref{eq:prin curv}, we can check \eqref{eq:W cubic +1} easily.

We proceed in the same way with the one-sheet hyperboloid of revolution
\begin{equation}\label{eq:1hyperboloid}
\frac{x^2+y^2}{a^2}-\frac{z^2}{b^2}=1,
\end{equation}
obtaining, from the generatrix hyperbola $x=a \cosh t $, $z= b \sinh t$, $t \in \R$, and Remark \ref{no ppa}, that
\begin{equation}\label{eq:K 1hyperboloid}
\mathcal K (x) = \frac{bx}{\sqrt{(a^2-b^2)x^2-a^4}}
\end{equation} 
and now we can check that the one-sheet hyperboloid of revolution satisfies the relation 
\begin{equation}\label{eq:W cubic -}
k_{\text m}=-\frac{a^4}{b^2} k_{\text p}^3.
\end{equation} 

However, for the two-sheets hyperboloid of revolution
\begin{equation}\label{eq:2hyperboloid}
-\frac{x^2+y^2}{a^2}+\frac{z^2}{b^2}=1,
\end{equation}
as now the generatrix hyperbola is $x=a \sinh t $, $z= \pm b \cosh t$, $t \geq 0$,  we get from Remark \ref{no ppa} that
\begin{equation}\label{eq:K 2hyperboloid}
\mathcal K (x) = \frac{\pm bx}{\sqrt{a^4+(a^2-b^2)x^2}}
\end{equation} 
and so we arrive at the following relation satisfied by the two-sheets hyperboloid of revolution:
\begin{equation}\label{eq:W cubic +2}
k_{\text m}=\frac{a^4}{b^2}\, k_{\text p}^3,
\end{equation}
that it is formally the same than the one \eqref{eq:W cubic +1} of the ellipsoid of revolution.

Finally, for the paraboloid of revolution
\begin{equation}\label{eq:paraboloid}
z=\frac{x^2+y^2}{2a},
\end{equation}
we use the generatrix parabola $x=t$, $z=\frac{t^2}{2a}$, $t\geq 0$, and Remark \ref{no ppa} gives that
\begin{equation}\label{eq:K paraboloid}
\mathcal K (x)= \frac{x}{\sqrt{a^2+x^2}}
\end{equation}
and then the paraboloid of revolution satisfies the relation 
\begin{equation}\label{eq:W cubic}
k_{\text m}=a^2 k_{\text p}^3.
\end{equation}

Now we are in a position to state our main result in this section characterizing all the quadric surfaces of revolution (see Figure \ref{fig:Paraboloids})
in terms of a cubic Weingarten relation.

\begin{figure}[h!]
\begin{center}
\includegraphics[height=4.5cm]{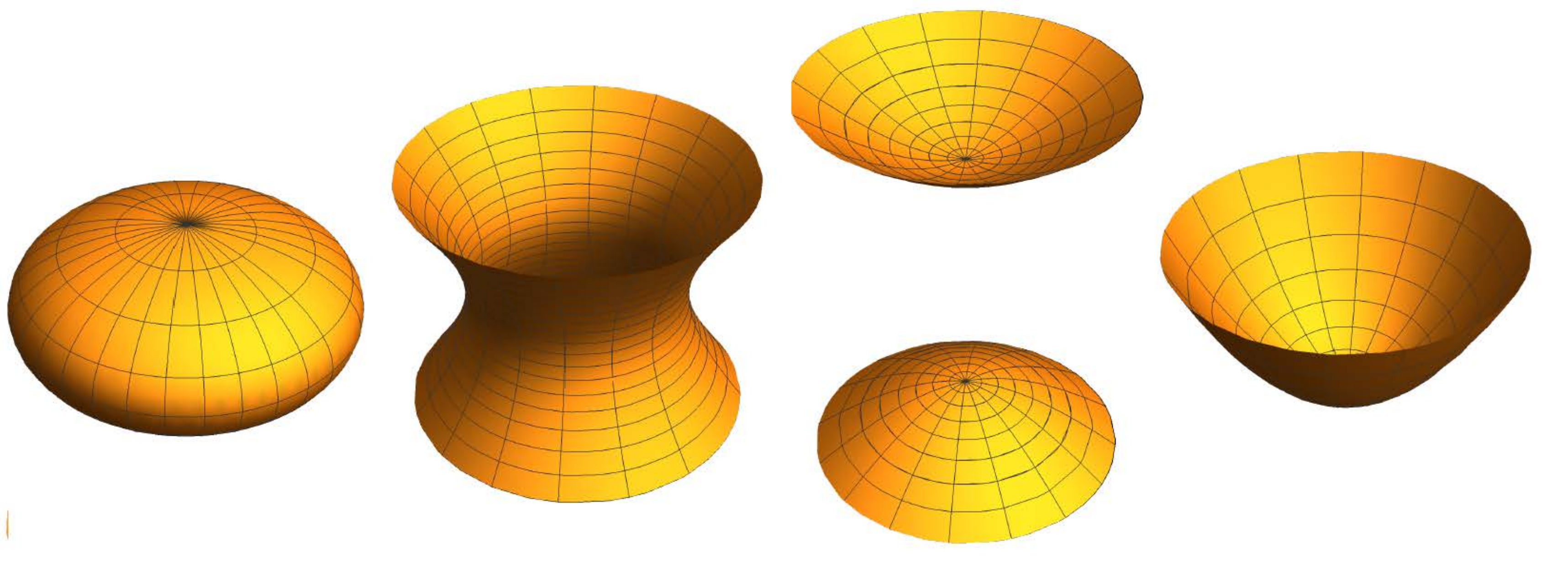}
\caption{Quadric surfaces of revolution, From left to right: Ellipsoid, one-sheet hyperboloid, two-sheets hyperboloid, paraboloid.}
\label{fig:Paraboloids}
\end{center}
\end{figure}

\begin{theorem}\label{th:W-cubic}
The only rotational surfaces satisfying $k_{\text m}=\mu\,  k_{\text p}^3$, $\mu \neq 0$, are the plane and the non-degenerate quadric surfaces of revolution.
\end{theorem}

\begin{proof}
Using \eqref{eq:prin curv}, the cubic Weingarten relation $k_{\text m}=\mu\,  k_{\text p}^3$ translates into the separable o.d.e.
$$\mathcal K ' (x)=\mu\, \mathcal K(x)^3/x^3. $$
Its constant solution $\mathcal K \equiv 0$ leads to the plane (see Proposition \ref{ex:rot surfaces}(1)).
Its non constant solution is given by
\begin{equation}\label{eq:K cubic}
\mathcal{K}(x)=\pm  \frac{x}{\sqrt{\mu + c\,x^2}}  , \, c\in\R.
\end{equation}
We are going to identify the rotational surfaces uniquely determined, up to $z$-translations, by the one parameter family of geometric linear momenta (depending on $c$) given in \eqref{eq:K cubic} (see Corollary \ref{cor:deter}). There is no restriction if we only consider plus sign in \eqref{eq:K cubic}.

We distinguish two cases according to the sign of $\mu$.
\begin{itemize}
\item $\mu >0$. We separate in turn three possibilities:
\begin{itemize}
\item[(i)] $c<1$: Then $a^2=\frac{\mu}{1-c}$ is well defined, and putting $b^2=\frac{a^4}{\mu}$, we get that \eqref{eq:K cubic} is exactly \eqref{eq:K ellipsoid} and Corollary \ref{cor:deter} gives that we arrive at the ellipsoid of revolution \eqref{eq:ellipsoid}. In particular, if $c=0$, we obtain the sphere of radius $\sqrt \mu$.
\item[(ii)] $c>1$: We now define $a^2=\frac{\mu}{c-1}$ and $b^2=\frac{a^4}{\mu}$. Then we obtain that \eqref{eq:K cubic} is exactly \eqref{eq:K 2hyperboloid} and Corollary \ref{cor:deter} concludes that we arrive at the two-sheets hyperboloid of revolution \eqref{eq:2hyperboloid}.
\item[(iii)] $c=1$: We define $a^2=\mu$ and we have that \eqref{eq:K cubic} is exactly \eqref{eq:K paraboloid}. We deduce from Corollary \ref{cor:deter} that we arrive at the paraboloid of revolution \eqref{eq:paraboloid}.
\end{itemize}
\item $\mu <0$. Since $\mu + c\,x^2>0 $ from \eqref{eq:K cubic} and taking into account that always $\mathcal K(x)^2 <1$, we deduce that $(1-c)x^2<\mu<0$ and so $c>1$. Then $a^2=\frac{\mu}{1-c}$ is well defined, and putting $b^2=-\frac{a^4}{\mu}$, we get that \eqref{eq:K cubic} is exactly \eqref{eq:K 1hyperboloid} and Corollary \ref{cor:deter} gives that we arrive at the one-sheet hyperboloid of revolution \eqref{eq:1hyperboloid}. 
\end{itemize}
This proves the result.
\end{proof}

\begin{remark}\label{german}
We point out that, although it is not explicitly established in \cite{KS05} because the authors were interested in closed Weingarten surfaces, the case $\mu>0$ in Theorem \ref{th:W-cubic} was considered in \cite{KS05}. 
\end{remark}

\subsection{Rotational Weingarten surfaces generated by elastic curves}\label{Sect elastic}
The elastic curves are those plane curves whose curvature is, at all points, proportional to the distance to a fixed line, called the directrix.
They were studied by Jacques Bernoulli in 1691 who named it elastica, by Euler in 1744, and Poisson in 1833 (cf.\ \cite{F93}).
With the elastica in the $xz$-plane and the directrix as the $z$ axis, the above condition can be written as $\kappa (x) = 2a x$, $a>0$, and then we have a one-parameter family of elastic curves determined, up to $z$-translations, by the geometric linear momenta
\begin{equation}\label{eq:K elasticoids}
\mathcal K (x)= ax^2 - k, \ k>-1,
\end{equation}
since $\mathcal K (x)^2<1$ (see Remark \ref{c}). We define the \textit{elasticoids} as the rotational surfaces generated by the rotation of an elastic curve around its directrix. 
Using Corollary \ref{cor:deter}, they are uniquely determined, up to translations along $z$-axis, by the geometric linear momenta given in \eqref{eq:K elasticoids}. Following \cite{F93} or \cite{S08} for example, we can distinguish seven types of elasticoids according to the seven types of elastic curves depending on the possible ranges of values of the \textit{modulus} $k\in (-1,+\infty) $ of the elliptic functions which appear in the parametrizations of the elastic curves generating the elasticoids.
See Figures \ref{fig:Elasticoids I}, \ref{fig:Elasticoids II}, \ref{fig:Elasticoids III} and \ref{fig:Elasticoids IV} for a description of them.

\begin{figure}[h!]
\begin{center}
\includegraphics[height=4.5cm]{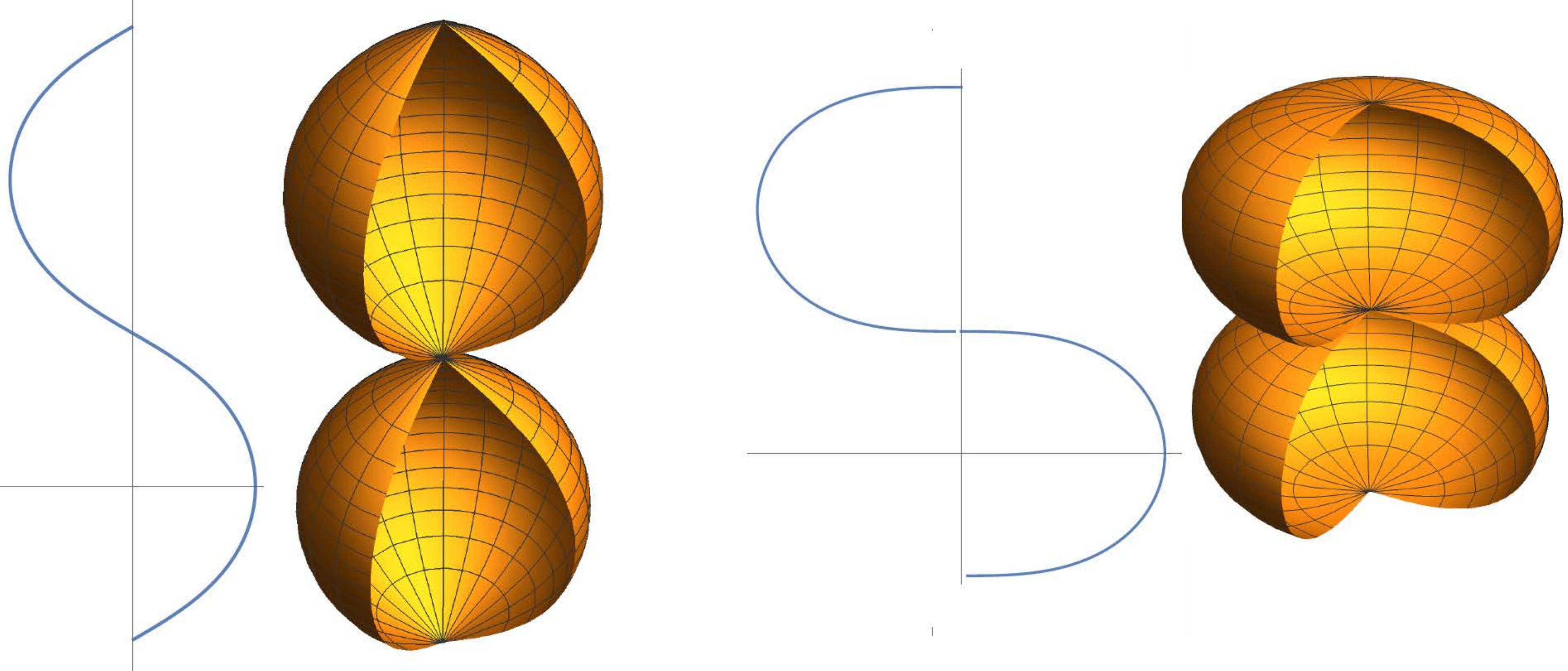}
\caption{Elasticoids I: surfaces of revolution generated by, from left to right, the pseudo-sinusoids ($-1<k<0$) and the right \textit{lintearia} ($k=0$).}
\label{fig:Elasticoids I}
\end{center}
\end{figure}

\begin{figure}[h!]
\begin{center}
\includegraphics[height=4.5cm]{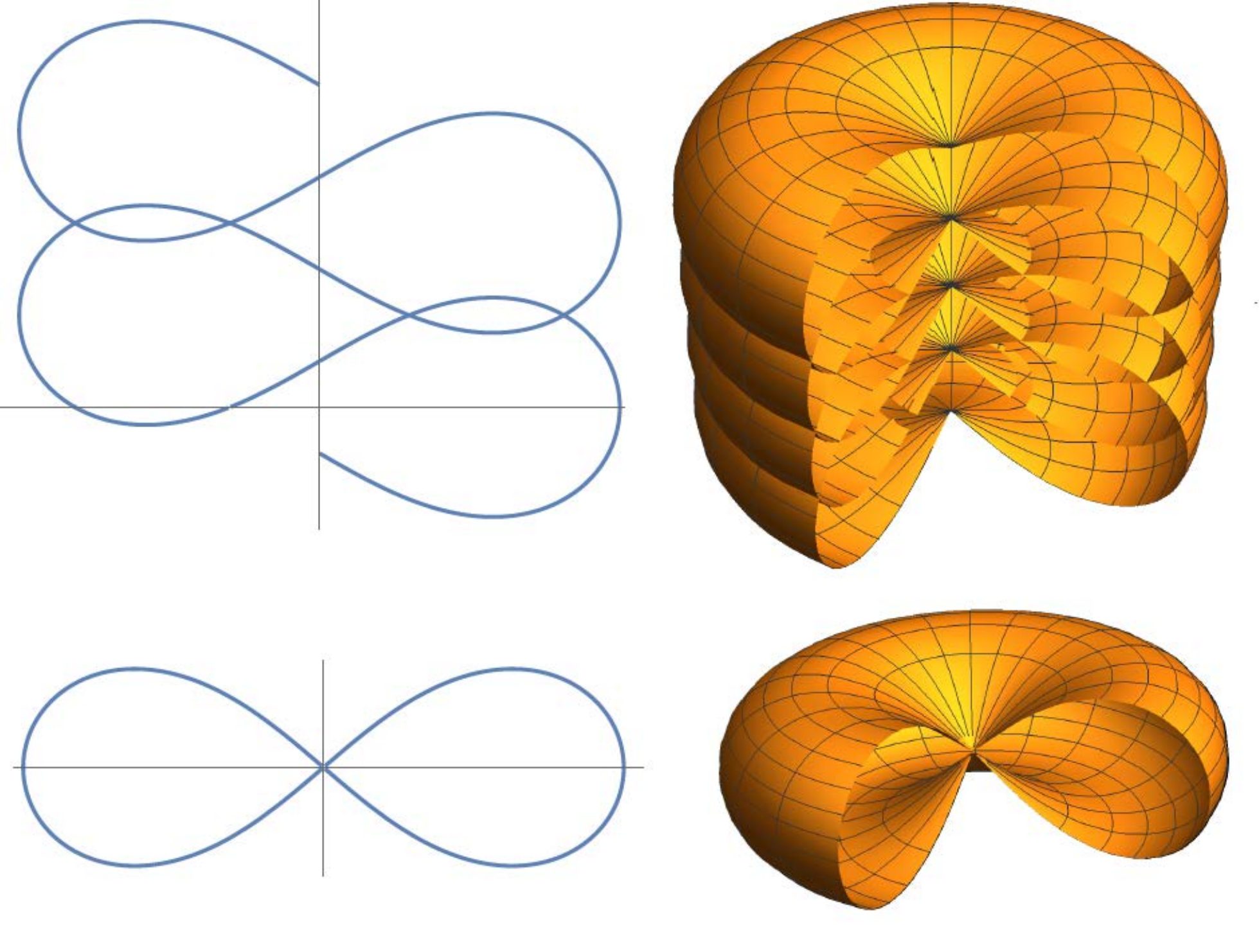}
\caption{Elasticoids II: surfaces of revolution generated by, from left to right, the elastic curves with $0<k<k_1$ and the pseudolemniscate ($k=k_1:=0.65222\dots$).}
\label{fig:Elasticoids II}
\end{center}
\end{figure}

\begin{figure}[h!]
\begin{center}
\includegraphics[height=4.5cm]{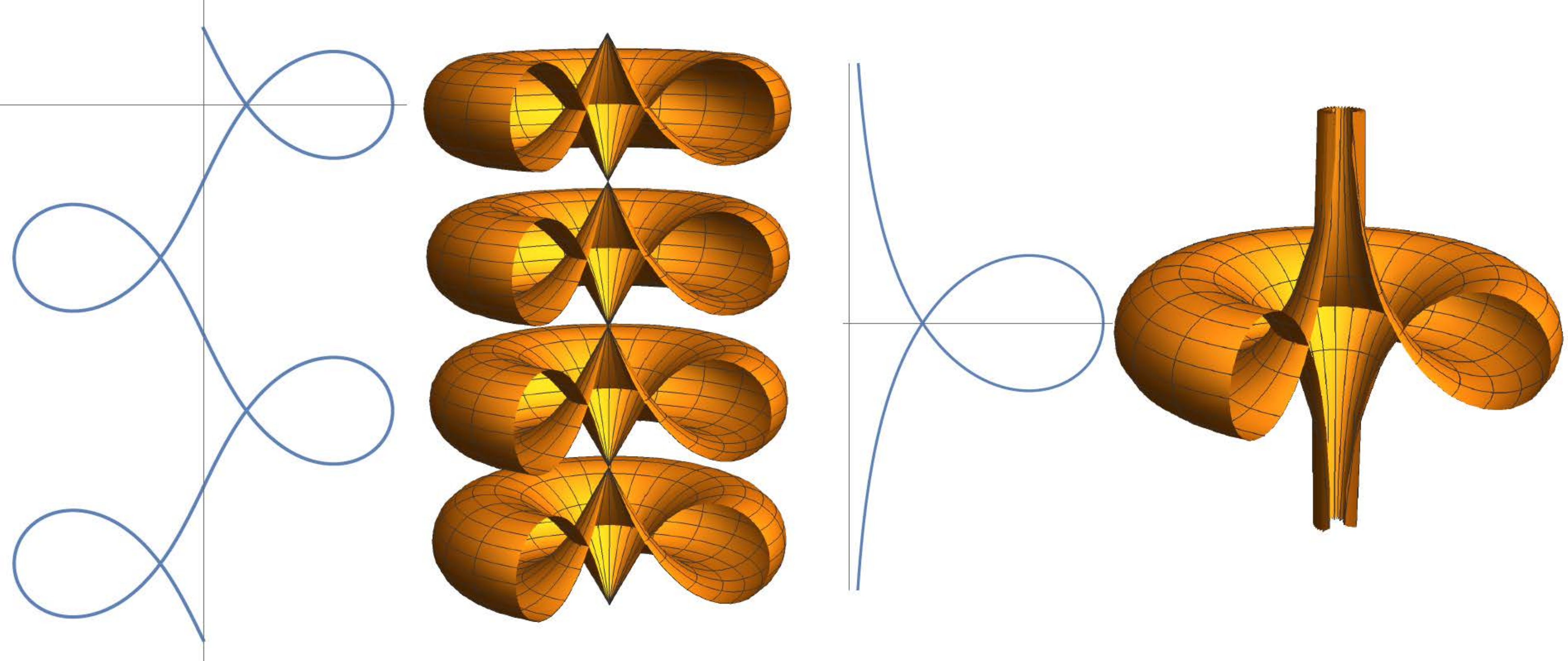}
\caption{Elasticoids III: surfaces of revolution generated by, from left to right, the elastic curves with $k_1<k<1$ and the convict curve ($k=1$).}
\label{fig:Elasticoids III}
\end{center}
\end{figure}

\begin{figure}[h!]
\begin{center}
\includegraphics[height=4.5cm]{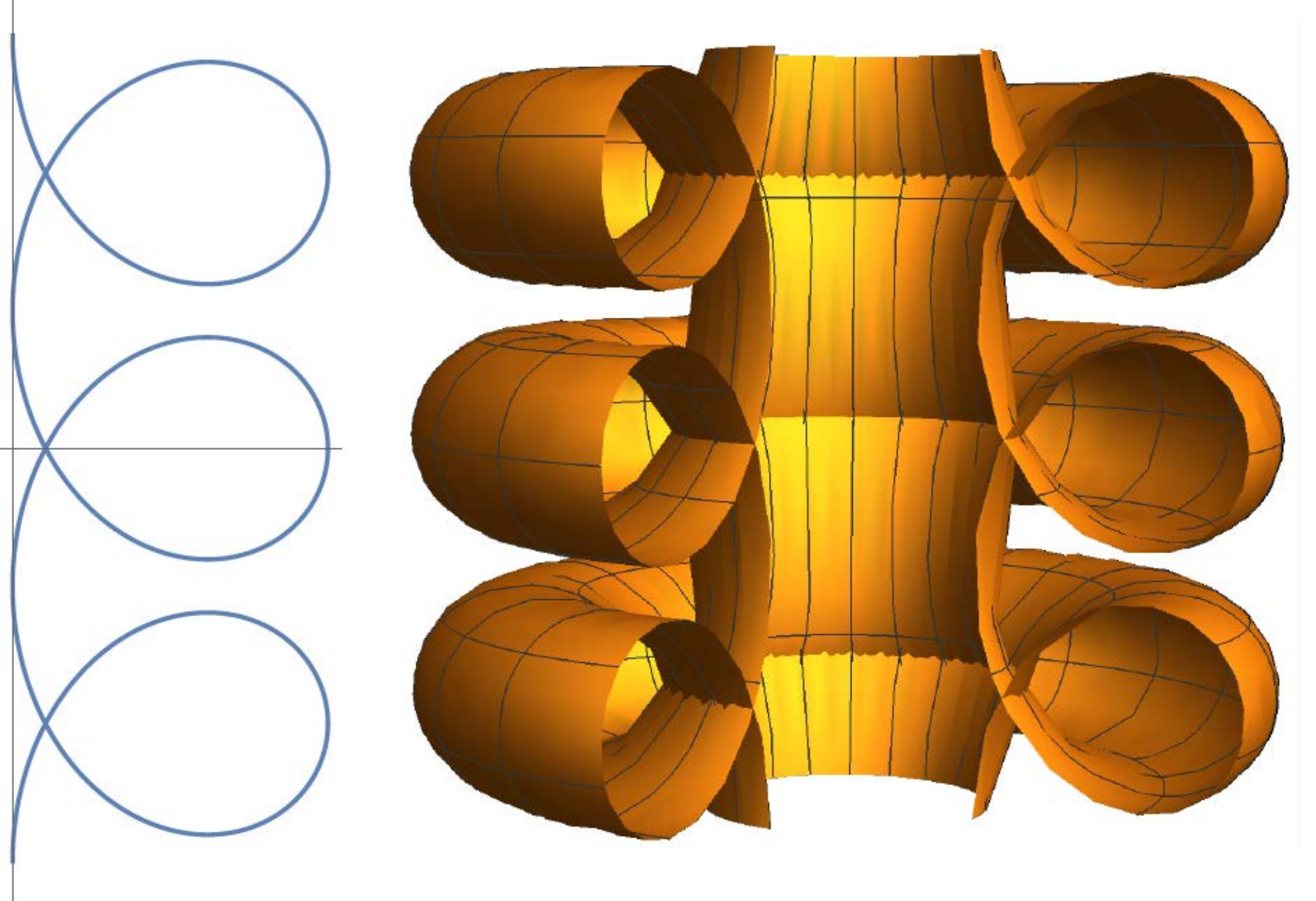}
\caption{Elasticoids IV: surfaces of revolution generated by the pseudotrochoids ($k>1$).}
\label{fig:Elasticoids IV}
\end{center}
\end{figure}

Using \eqref{eq:prin curv} and \eqref{eq:K elasticoids}, we have that the principal curvatures for the elasticoids are given by
$k_{\text m}= 2 a x$ and $k_{\text p}= ax -k/x$. We can eliminate $x$ since $a\neq 0$ obtaining the W-diagram of the elasticoids:
\begin{equation}\label{eq:W_elasticoids}
k_{\text m}^2-2 k_{\text p}\, k_{\text m} = 4ak.
\end{equation}
If $k\neq 0$, they are hyperbolae with asymptotes $k_{\text m}=0$ and $k_{\text m}=2 k_{\text p}$ (see Figure \ref{fig:W-diagram}), which are obtained in \eqref{eq:W_elasticoids} when $k=0$.
\begin{figure}[h!]
\begin{center}
\includegraphics[height=3.5cm]{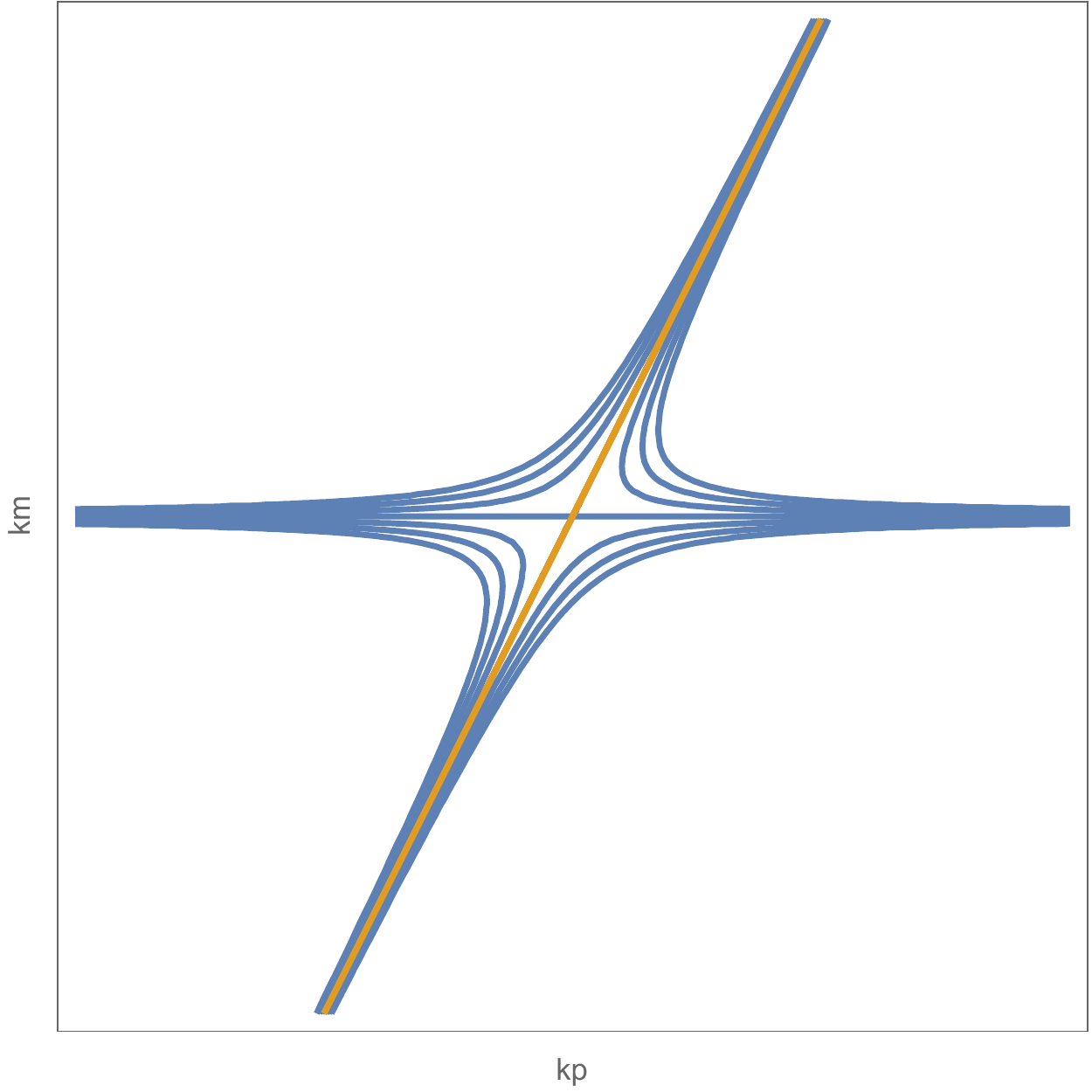}
\caption{W-diagram of the elasticoids.}
\label{fig:W-diagram}
\end{center}
\end{figure}

The elasticoids with null modulus, i.e.\ $k=0$, are linear Weingarten surfaces since they satisfy $k_{\text m}=2 k_{\text p}$ (see Subsection \ref{Sect linear}). Otherwise, we provide the following uniqueness result for Weingarten surfaces involving hyperbolae in the curvature diagram.

\begin{theorem}\label{th:W-hyperbola}
The only rotational surfaces satisfying $k_{\text m}^2-2 k_{\text p}\, k_{\text m} + \mu=0$, $\mu \neq 0$, are the sphere of radius $R=1/\mu$ ($\mu >0$) and the elasticoids with nonzero modulus.
\end{theorem}

\begin{proof}
Using \eqref{eq:prin curv}, the Weingarten relation $k_{\text m}^2-2 k_{\text p}\, k_{\text m} + \mu=0$ translates into the o.d.e.
\begin{equation}\label{eq:ode K elastic}
\mathcal K ' (x)^2 -2 \mathcal K(x) \mathcal K ' (x) /x +\mu =0.
\end{equation}
We make the change of variable $\lambda (x)=\mathcal K  (x)/x$ and \eqref{eq:ode K elastic} becomes into the separable o.d.e. for $\lambda$:
\begin{equation}\label{eq:ode kp elastic}
 \lambda' (x)=\frac{\sqrt{\lambda^2(x)-\mu}}{x}.
\end{equation}
The constant solution $\lambda^2\equiv\mu>0$ leads to $\mathcal K(x)=\sqrt \mu\, x$ and Proposition \ref{ex:rot surfaces}(3) gives the sphere of radius $R=1/\mu$. Otherwise, by integrating \eqref{eq:ode kp elastic}, we get $\lambda (x)+\sqrt{\lambda^2(x)-\mu}=2ax$, for some constant $a>0$. From here we easily deduce that $\mathcal K(x)=ax^2+\frac{\mu}{4a}$. Looking at \eqref{eq:K elasticoids}, using Corollary \ref{cor:deter}, we arrive at the elasticoid of modulus
$k=-\frac{\mu}{4a}$, which is consistent with \eqref{eq:W_elasticoids}. This finishes the proof. 
\end{proof}

%If $\alpha=(x,z)$ is a elastic curve in the $xz$-plane, its curvature satisfies $\kappa (x) = 2 a x + b$, $a\neq 0$, $b\in \R$ (see Remark \ref{difficulties}(1)). Then we have the corresponding one-parameter family of geometric linear momenta given by $\mathcal K (x) = ax^2 +b x +c$, $c\in \R$.
%The constant $c$ distinguishes precisely the different kinds of elasticae (wavelike, borderline, orbitlike, \dots) by their relative position with respect to the $z$-axis (see \cite[Section 3]{CCI16}). Using \eqref{eq:prin curv}, we have that the principal curvatures for the rotational surface $S_\alpha$ are given by
%$k_{\text m}= 2 a x$ and $k_{\text p}= ax + b +c/x$. We can eliminate $x$ since $a\neq 0$ obtaining the quadratic Weingarten relation satisfied by a general elastic curve:
%\begin{equation}\label{eq:W_elastic}
%k_{\text m}^2-2 k_{\text m} k_{\text p}+2b\, k_{\text m}+4ac=0.
%\end{equation}
%If $c\neq 0$, it represents an hyperbola in the W-diagram With asymptotes $k_{\text m}=0$ and $k_{\text m}=2(k_{\text p}-b)$. 
%Precisely, when we are dealing with free elasticae, i.e.\ $c=0$, from \eqref{eq:W_elastic} we arrive at special linear Weingarten surfaces satisfying
%$k_{\text m}=2(k_{\text p}-b)$. If, in addition, $b=0$ then we get $k_{\text m}=2 k_{\text p}$. 

\section{Prescribing curvature on a rotational surface}\label{Sect4}
Making use of Corollary \ref{cor:deter}, we present in this section an existence and uniqueness result on prescribing mean or Gauss curvature for a rotational surface with an arbitrary continuous function depending on the distance from the surface to the axis of revolution.

\begin{theorem}\label{th:Prescribe H KGauss}
\noindent
\begin{enumerate}
\item[(a)] Let $H=H(x)$, $x>0$, be a continuous function. 
Then there exists a one-parameter family of rotational surfaces with mean curvature $H(x)$, 
$x$ being the distance from the surface to the axis of revolution.
The surfaces in the family are uniquely determined, up to translations along $z$-axis, by the geometric linear momenta of their generatrix curves given by
\begin{equation}\label{eq:mom-H}
x \, \mathcal K (x) =2 \! \int \! x  H(x) dx.
\end{equation}
\item[(b)] Let $K_{\text{G}}=K_{\text{G}}(x)$, $x>0$, be a continuous function. 
Then there exists a one-parameter family of rotational surfaces with Gauss curvature  $K_{\text{G}}(z)$,
$x$ being the distance from the surface to the axis of revolution. The surfaces in the family are uniquely determined, up to translations along $z$-axis, by the geometric linear momenta of their generatrix curves given by
\begin{equation}\label{eq:mom-KGauss}
 \mathcal K(x)^2 =2 \! \int \! x  K_{\text G}(x) dx.
\end{equation}
\end{enumerate}
\end{theorem}

\begin{remark}\label{cte int}
The parameter in both uniparametric families described in Theorem \ref{th:Prescribe H KGauss} comes from the integration constant in \eqref{eq:mom-H} and \eqref{eq:mom-KGauss}.
\end{remark}

\begin{proof}
The proof of part (a) simply relies on solving the linear equation
$$ \mathcal K ' (x)+\frac{\mathcal K (x)}{x}=2H(x) $$
for unknown $\mathcal K $, coming from \eqref{eq:H}, and applying Corollary \ref{cor:deter}.

The same reasoning gives part (b) solving now the immediate equation 
$$\mathcal K ' (x) \mathcal K (x)=x K_{\text G}(x)$$
coming from \eqref{eq:KGauss}.
\end{proof}

As a first application of part (a) in Theorem \ref{th:Prescribe H KGauss}, we provide a very short simple proof of a classical result of Euler (cf.\ \cite{E44}) concerning rotational minimal surfaces, i.e.\ those with vanishing mean curvature.

\begin{corollary}\label{cor:Euler Th}
The only minimal rotational surfaces are the plane and the catenoid.
\end{corollary}

\begin{proof}
If $H=0$, then \eqref{eq:mom-H} leads to $ \mathcal K(x)=c/x$, $c\in \R$.
When $c=0$, Proposition \ref{ex:rot surfaces}(1) gives the plane, and when $c\neq 0$, 
Proposition \ref{ex:rot surfaces}(5) leads to the catenoid. 
\end{proof}

Applying again part(a) in Theorem \ref{th:Prescribe H KGauss}, we now deal with a new shorter proof of the classification of rotational CMC surfaces given by Delaunay (cf.\ \cite{D41}, see also \cite{PM19}). We recall (see \cite{F93}) that the {\em Delaunay surfaces} are the surfaces of revolution generated by the rotation of the {\em Delaunay roulettes} around their base. According to \cite{F93}, the differential equation of these curves in the $xz$-plane is given by 
\begin{equation}\label{eq: roulettes}
\left( \frac{dx}{dz} \right)^2 = \frac{4a^2 x^2-(x^2+\epsilon b^2)^2}{(x^2+\epsilon b^2)^2},
\end{equation}
with $\epsilon = 1$ for the elliptic roulette (ellipse with semi-axes $a$ and $b$, $a > b$), 
and $\epsilon = -1$ for the hyperbolic roulette (hyperbola with semi-axes $a$ and $b$).
The surfaces associated to an elliptic roulette are called {\em onduloids} and those ones associated to a hyperbolic roulette {\em nodoids}.
Since the parabolic roulette is the catenary, the associated surface is none other than the catenoid (case of zero mean curvature).
The Delaunay surfaces have the remarkable property, apart from the circular cylinder (corresponding to a circular roulette), of being the only surfaces of revolution with nonzero constant mean curvature (if we include the sphere, which is a limit case). See Figure \ref{fig:Roulettes}.

\begin{figure}[h!]
\begin{center}
\includegraphics[height=2.5cm]{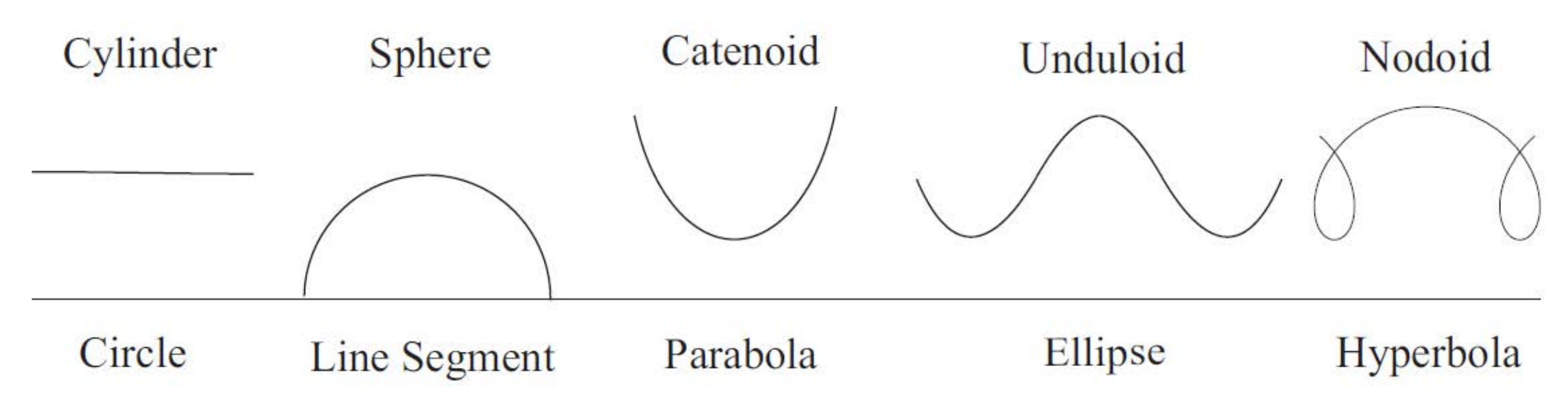}
\caption{Delaunay roulettes (the profile curves of the Delaunay surfaces)
generated by rolling the conics listed below them; this is Figure 1 in \cite{PM19}.}
\label{fig:Roulettes}
\end{center}
\end{figure}

\begin{corollary}\label{cor:Delaunay Th}
The only nonzero constant mean curvature rotational surfaces are the circular cylinder, the sphere and the Delaunay surfaces.
\end{corollary}

\begin{proof}
Assume, without restriction, $H\equiv H_0>0$.
Then \eqref{eq:mom-H} leads to 
\begin{equation}\label{eq:K H0}
\mathcal K(x)=H_0 x + c/x, \ c\in \R.
\end{equation}
If $c=0$, $\mathcal K(x)=H_0 x$ and we get the sphere of radius $\frac{1}{H_0}$ from Proposition \ref{ex:rot surfaces}(3).
We must also take into account the case $x$ constant, because it produces the circular cylinder of radius $\frac{1}{2H_0}$ (see Remark \ref{cylinder}).
Being $x$ non constant, we look for the generatrix curve determined by \eqref{eq:K H0} applying Remark \ref{c}, and obtain:
$$
s=s(x)=\int \frac{x dx}{\sqrt{x^2-(H_0 x^2+c)^2}}
$$
and
$$
z=z(s)=\int \left(H_0 x(s)+\frac{c}{x(s)}\right)ds.
$$
From the above two expressions, we deduce that
$$
\left( \frac{dz}{dx} \right)^2 = \frac{(H_0 x^2+c)^2}{x^2-(H_0 x^2+c)^2}
$$
Comparing with \eqref{eq: roulettes}, we arrive at Delaunay roulettes considering 
$a=\frac{1}{2H_0}$ and $\epsilon b^2=\frac{c}{H_0}$, that is, if $c>0$, taking $\epsilon =1 $ and so $b^2=c/H_0$ and if $c<0$, taking $\epsilon =-1 $ and so $b^2=-c/H_0$. This finishes the proof.
\end{proof}

Now we deal with some applications of part (b) in Theorem \ref{th:Prescribe H KGauss}.
It is immediate to get the flat rotational surfaces, since $K_{\text G}\equiv 0$ implies that $\mathcal K$ must be constant, and we arrive at planes, cones and cylinders (see Proposition \ref{ex:rot surfaces}(1),(2) and Remark \ref{cylinder}). More interesting is the study (initiated by Darboux in 1890) of rotational surfaces with nonzero constant Gauss curvature. They are called \textit{Darboux surfaces} and we present them following the description in \cite{F93} and references therein.

Let $K_{\text G}\equiv K_0\neq 0$.
\begin{itemize}
\item[(i)] Positive Gaussian curvature $K_0=1/a^2$, $a>0$.

In cylindrical coordinates $(r,\theta,z)$, the surfaces are defined by
\begin{equation}\label{eq:Kpos}
r=a k \cos t, \quad 
z=a E(k,t):=a \int_0^t \sqrt{1-k^2 \sin^2 u}\, du, \ k>0,
\end{equation} 
where $E(k,\cdot)$ denotes the elliptic integral of second kind and modulus $k>0$ (\cite{BF71}).
The case $k=1$ correspond to the sphere and the shapes are different according to $0<k<1$ or $k>1$. See Figure \ref{fig:Kctepos}.

\begin{figure}[h!]
\begin{center}
\includegraphics[height=4.5cm]{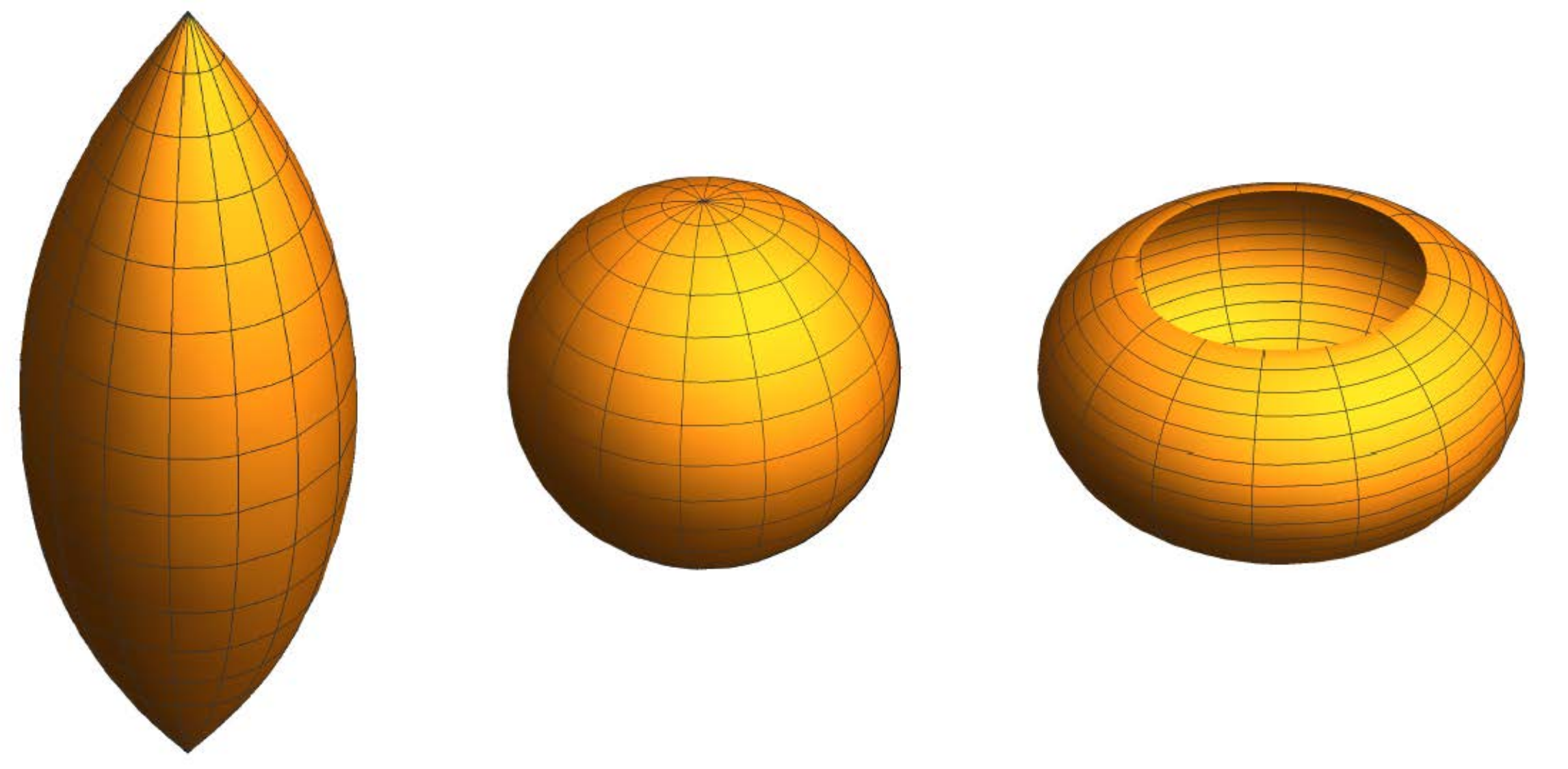}
\caption{Rotational surfaces with $K_{\text G}\equiv K_0> 0$.
 From left to right: $0<k<1$, $k=1$ (sphere), $k>1$.}
\label{fig:Kctepos}
\end{center}
\end{figure}

\item[(ii)] Negative Gaussian curvature $K_0=-1/a^2$, $a>0$.

There are also three kinds of surfaces, but this time with three different parametrizations.
See Figure \ref{fig:Kcteneg}.
\begin{itemize}
\item First type: surfaces with a conical point. In cylindrical coordinates, the surfaces are given by
\begin{equation}\label{eq:KnegI}
r=a k \sinh t, \quad z=a \int_0^t \sqrt{1-k^2 \cosh^2 u}\, du, \ 0<k<1.
\end{equation}
\item Second type: surfaces that look like a hyperboloid.
In cylindrical coordinates, the surfaces are described by
\begin{equation}\label{eq:KnegII}
r=a k \cosh t, \quad z=a \int_0^t \sqrt{1-k^2 \sinh^2 u}\, du , \ k\neq 0.
\end{equation}
\item Third type: the \textit{pseudosphere}. The pseudosphere is the surface of revolution generated by the rotation of a tractrix around its asymptote (see Proposition \ref{ex:rot surfaces}(7)). It was studied by Ferdinand Minding (1806-1885) and Eugene Beltrami in 1868.
\end{itemize}
\end{itemize}

\begin{figure}[h!]
\begin{center}
\includegraphics[height=4.5cm]{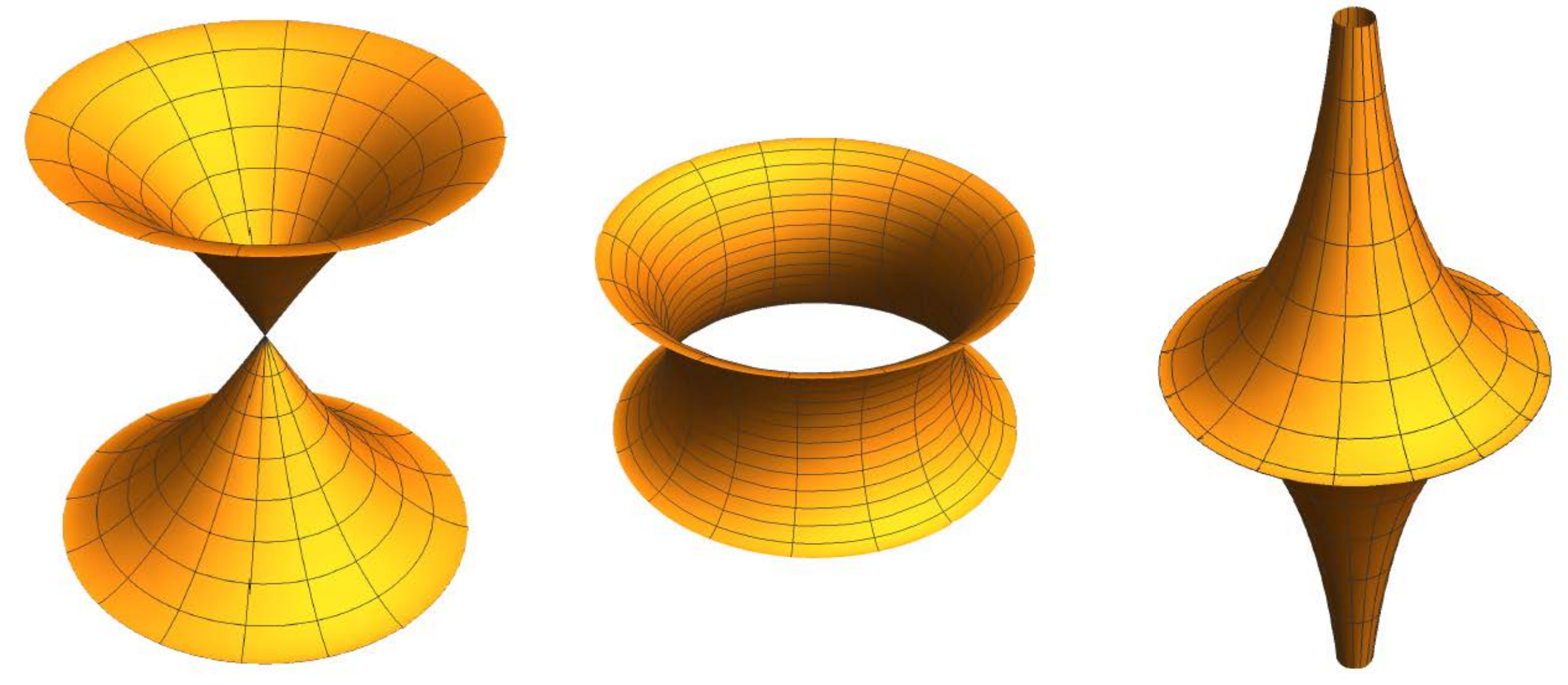}
\caption{Rotational surfaces with $K_{\text G}\equiv K_0< 0$.
 From left to right: surface with a conical point, surface of hyperboloid-type, pseudosphere.}
\label{fig:Kcteneg}
\end{center}
\end{figure}

Applying part(b) in Theorem \ref{th:Prescribe H KGauss}, we provide a new proof of the classification of rotational surfaces with nonzero constant Gaussian curvature.

\begin{corollary}\label{cor:Darboux Th}
The only nonzero constant Gauss curvature rotational surfaces are the Darboux surfaces.
\end{corollary}

\begin{proof}
Assume $K_{\text G}\equiv K_0\neq 0$.
Then \eqref{eq:mom-KGauss} leads to 
\begin{equation}\label{eq:K K0}
\mathcal K (x)^2 = K_0 x^2 + c, \ c\in \R.
\end{equation}
If $c=0$, $\mathcal K(x)^2=K_0 x^2$, $K_0>0$, and we get the sphere of radius $\frac{1}{\sqrt{K_0}}$ from Proposition \ref{ex:rot surfaces}(3).

Now we look for the rotational surface determined by \eqref{eq:K K0} (see Corollary \ref{cor:deter}). Applying \eqref{eq:zgraph}, we obtain:
%$$ s=s(x)=\frac{dx}{\sqrt{1-c-K_0 x^2}} $$
%$$ z=z(s)= \int \sqrt{c+K_0 x(s)^2} ds $$
\begin{equation}\label{eq:zKcte}
z=z(x)=\pm \int \frac{\sqrt{c+K_0 x^2}}{\sqrt{1-c-K_0 x^2}}.
\end{equation}

We must distinguish two cases according to the sign of $K_0$.
If $K_0=1/a^2>0$, we have that
\begin{equation}\label{eq:zKpos}
z=z(x)=\pm \int \frac{\sqrt{a^2c + x^2}}{\sqrt{a^2(1-c)- x^2}},
\end{equation}
with $c<1$ in this case. We put $c=1-k^2$, $k\in \R$, and
make the change of variable $x=ak \cos t $ in \eqref{eq:zKpos}. Then we arrive at $z=\mp a \int \sqrt{1-k^2 \sin^2 t}\, dt$. Looking at \eqref{eq:Kpos}, we recover the Darboux surfaces with positive constant Gauss curvature. We point out that $k=1$ (and then $c=0$) gives again the sphere of radius $a>0$ as then $z(x)=\mp \sqrt{a^2-x^2}$.

If $K_0=-1/a^2$, from \eqref{eq:zKcte} we get that 
\begin{equation}\label{eq:zKneg}
z=z(x)=\pm \int \frac{\sqrt{a^2c - x^2}}{\sqrt{a^2(1-c)+ x^2}},
\end{equation}
with $c>0$ in this case.

When $0<c<1$, we write $c=1-k^2$, $0<k<1$. 
Then we make the change of variable
$x=ak \sinh t $ in \eqref{eq:zKneg} and so we arrive at $z=\pm a \int \sqrt{1-k^2 \cosh^2 t}\, dt$. Looking at \eqref{eq:KnegI}, we recover the first type of Darboux surfaces with negative constant Gauss curvature.

When $c>1$, we can set $c=1+k^2$, $k\neq 0$.
Now we make the change of variable
$x=ak \cosh t $ in \eqref{eq:zKneg} and arrive at $z=\pm a \int \sqrt{1-k^2 \sinh^2 t}\, dt$. Looking at \eqref{eq:KnegII}, we recover the second type of Darboux surfaces with negative constant Gauss curvature.

When $c=1$, we simply have that $$z=z(x)=\pm \int \frac{\sqrt{a^2-x^2}}{x}dx=\pm \left( \sqrt{a^2-x^2}-a \ln \left( \frac{a+\sqrt{a^2-x^2}}{x} \right) \right),  $$
that is nothing but the Cartesian equation of the tractrix (cf.\ \cite{F93}), the generatrix curve of the pseudosphere. Anyway, from \eqref{eq:K K0} we deduce that $\mathcal K(x) = \sqrt{1-x^2/a^2}$ in this case. Then Proposition \ref{ex:rot surfaces}(7) also leads to the pseudosphere. This finishes the proof.

\end{proof}

\section*{Acknowledgements}
%The second author is partially supported by State Research Agency and European Regional Development Fund via the grant \ 2020-117868GB-I00 supported by MCIN / AEI / 10.13039 / 501100011033.

This publication is part of the R+D+i project PID2020-117868GB-I00,
financed by MCIN / AEI / 10.13039 / 501100011033 /. The second author is also supported by the EBM / FEDER UJA 2020 project 1380860.

%Second named author:
%Ayuda Referencia de la ayuda María de Maeztu
%supported by MCIN/AEI/10.13039/501100011033.

\end{document}